%% file: tensorization_multivariate_arxiv_v2.tex
\documentclass[11pt, a4paper]{amsart}

\usepackage{amssymb, bbm}
\usepackage{chngcntr}
\usepackage[utf8]{inputenc}
\usepackage{tikz}
\usepackage{subcaption}
\usepackage[normalem]{ulem}

\usepackage{fancyhdr}
\usepackage{amsaddr}
\usepackage{color}
\usepackage{enumitem}
\usepackage{dsfont}
\usepackage{xargs}
\usepackage{tikz}
\usetikzlibrary{calc,shapes,snakes,patterns,matrix,backgrounds,trees,arrows}
\usepackage{algorithm}
\usepackage{algpseudocode}

\usepackage[sort, nocompress]{cite}

\usepackage[margin=0.75in]{geometry}
\usepackage{proof-at-the-end}
\usepackage{hyperref}
\usepackage{cleveref}

\newtheorem{theorem}{Theorem}[section]

\newtheorem{lemma}[theorem]{Lemma}

\newtheorem{remark}[theorem]{Remark}
\newtheorem{definition}[theorem]{Definition}

\newtheorem{main result}[theorem]{Main Result}

\definecolor{pastelred}{rgb}{1.0, 0.41, 0.38}
\definecolor{lightgreen}{rgb}{0.56, 0.93, 0.56}
\definecolor{lightblue}{rgb}{0.53, 0.81, 0.98}

\counterwithin{equation}{section}
\counterwithin{theorem}{section}

\input{macros}

\newcommand{\rev}[1]{#1}
\newcommand{\revQ}[1]{#1}
\newcommand{\revY}[1]{#1}
\newcommand{\revZ}[1]{#1}

\title{Approximation Theory of Tree Tensor Networks: Tensorized Multivariate Functions}

\author{Mazen Ali \and Anthony Nouy}
\address{Nantes Universit\'e, Centrale Nantes, \\Laboratoire de Math\'ematiques Jean Leray, CNRS UMR 6629, France}
\email{mazen.ali90@gmail.com}
\email{anthony.nouy@ec-nantes.fr}

\thanks{Acknowledgments: This work was partially conducted within the France 2030 framework programme, Centre Henri Lebesgue ANR-11-LABX-0020-01. 
The authors acknowledge AIRBUS Group for the financial support with the project AtRandom. 
AN also acknowledges the partial funding by the ANR-DFG project COFNET (ANR-21-CE46-0015).
}
\date{\today}

\keywords{Tensor Networks, Tensor Trains, Matrix Product States,
Neural Networks, Approximation Spaces, Besov Spaces,
direct (Jackson) and inverse (Bernstein) inequalities. }
\subjclass[2010]{41A65, 41A15, 41A10 (primary); 68T05, 42C40, 65D99 (secondary)}

\begin{document}

\begin{abstract}
 We study the approximation of  multivariate functions  with tensor networks (TNs), \rev{providing some answers to  the following two questions:  ``\emph{what are the approximation capabilities of TNs for functions from classical  smoothness classes?}''    and    ``\emph{what are the properties of the  class of functions that can be  approximated with TNs with a certain performance?}''}
   
    \rev{As a partial answer to the former}, we show that TNs   can (near to) optimally replicate $h$-uniform and $h$-adaptive  spline  approximation, for any smoothness order    of the target function.    Tensor networks thus exhibit   universal \emph{expressivity} w.r.t.    isotropic, anisotropic and mixed smoothness spaces that is    comparable with more general neural networks families such as    deep rectified linear unit (ReLU) networks.    Put differently, TNs have the capacity to (near to) optimally approximate    many function classes -- without being adapted to the    particular class in question.
    
    \rev{As a partial answer to the latter}, as a candidate model class    we consider \emph{approximation classes} of TNs    and show that these    are (quasi-)Banach spaces, that many types of classical    smoothness spaces are continuously embedded    into said approximation classes and    that TNs approximation classes are themselves    not embedded in any classical smoothness space. \rev{In other words, TNs can efficiently approximate  functions that lie beyond classical smoothness spaces.}
\end{abstract}

\maketitle

\input{introduction}
\input{tensorization}
\input{tnetworks}
\input{appspaces}
\input{embeddings}

\input{conclusion}

\appendix
\input{review}
\section{Proofs}
\printProofs

\bibliographystyle{acm}
\bibliography{literature}

\end{document}

%% file: macros.tex
\newcommand{\bs}{\boldsymbol}
\newcommand{\mc}{\mathcal}

\renewcommand*\d{\mathop{}\!\mathrm{d}}
\DeclareMathOperator*{\linspan}{span}

\newcommand{\R}{\mathbb{R}}
\newcommand{\N}{\mathbb{N}}
\newcommand{\Z}{\mathbb{Z}}
\newcommand{\id}{\mathbb{I}}

\newcommand{\tbl}{\ensuremath{t^D_{b,L}}}
\newcommand{\Tbl}{\ensuremath{T^D_{b,L}}}
\newcommand{\Vbl}{\ensuremath{\bs{V}^D_{b,L}}}
\newcommand{\VblSDD}{\ensuremath{\bs{V}^D_{b,L, S^D}}}
\newcommand{\VblSD}{\ensuremath{\bs{V}^D_{b,L, S^{\otimes D}}}}
\newcommand{\Vblm}{\ensuremath{\bs{V}^D_{b,L, m}}}
\newcommand{\vblSD}{\ensuremath{{V}^D_{b,L, S^D}}}
\newcommand{\vblm}{\ensuremath{{V}^D_{b,L, m}}}
\newcommand{\Pm}{\ensuremath{\mathbb{P}_m}}
\newcommand{\Umin}{\ensuremath{\bs{U}^{\min}}}
\newcommand{\T}{\ensuremath{\mathcal{T}}}
\newcommand{\TT}{\ensuremath{\mathcal{TT}}}
\newcommand{\PbS}{\ensuremath{\mathcal{P}_{b, L, S, \bs r}}}
\newcommand{\RbS}{\ensuremath{\mathcal{R}_{b, L, S, \bs r}}}
\newcommand{\PhiblS}{\ensuremath{\Phi}_{b,L,S,\bs r}}
\newcommandx{\norm}[2][2=]{\ensuremath{\left\| #1 \right\|_{#2}}}
\newcommandx{\snorm}[2][2=]{\ensuremath{\left| #1 \right|_{#2}}}
\newcommand{\Asq}{\ensuremath{A^\alpha_q}}
\newcommand{\full}{\ensuremath{\mathrm{compl}_{\mc F}}}
\newcommand{\sparse}{\ensuremath{\mathrm{compl}_{\mc S}}}
\newcommand{\neuron}{\ensuremath{\mathrm{compl}_{\mc N}}}
\newcommand{\PhiF}{\ensuremath{\Phi^{\mc F}}}
\newcommand{\PhiS}{\ensuremath{\Phi^{\mc S}}}

\newcommand{\Fsq}{\ensuremath{\mathcal{F}^\alpha_q(L^p)}}
\newcommand{\Ssq}{\ensuremath{\mathcal{S}^\alpha_q(L^p)}}

\newcommand{\I}{\ensuremath{\mathrm{I}}}
\newcommand{\A}{\ensuremath{\mathrm{A}}}
\newcommand{\M}{\ensuremath{\mathrm{M}}}
\newcommand{\B}{\ensuremath{B_q^{s_\I}(L^p(\Omega))}}
\newcommand{\AB}{\ensuremath{AB_q^{\bs\alpha}(L^p(\Omega))}}
\newcommand{\MB}{\ensuremath{MB_q^{s_\M}(L^p(\Omega))}}
\newcommand{\indicator}{\ensuremath{\mathbbm{1}}}

%% file: introduction.tex

\section{Introduction}
We study the approximation of real-valued functions
$f:\Omega\rightarrow\R$ on bounded $D$-dimensional
domains $\Omega\subset\R^D$ using tensor networks (TNs).
This work is a continuation of
\cite{ali2023approximation}. 
We refer to \cite{ali2023approximation} for a more detailed introduction.

\subsection{Previous Work}
Originally, TNs were used to approximate algebraic tensors
$f\in\R^{n_1\times\ldots\times n_D}$,
see, e.g., \cite{Hackbusch2012}.
In \cite{Oseledets2009}, the author used tensor trains (TT) to approximate matrices
by writing the row- and column-indices in the binary form
$i_\nu=\sum_{k=1}^{L}j_\nu^k2^{L-k} = \rev{[j_\nu^1, \dots , j_\nu^L]_2}$. This way
a matrix can be written as a higher-order tensor
\begin{align*}
    f(i_1,i_2)=\tilde{f}(j_1^1,\ldots,j_1^L,j_2^1,\ldots,
    j_2^L),
\end{align*}
and the higher-order tensor $\tilde{f}$ can be approximated by TTs.
This was
later coined the \emph{Quantics Tensor Train (QTT)} \cite{Khoromskij2011}
and morphed into \emph{Quantized TT} over time. \rev{Other tensor formats such as Hierarchical Tucker format of more general tree tensor networks could be used as well,  after a possible permutation of indices.} 
Functions of real variables can also be identified with higher-order tensors, or \emph{tensorized}, by using a binary or more general 
$b$-adic encoding of real numbers.
Representations of polynomials with the QTT format were studied in
\cite{Grasedyck2010,Oseledets2012} 
and numerical approximation for PDEs in \cite{Kazeev2017}
(and references therein).

In  \cite{ali2023approximation},
we studied approximation theoretic properties of TNs
for one-dimensional ($D=1$) functions in various
smoothness spaces. In \revY{the present work},
we extend \revY{these} results to the multi-dimensional setting
$D>1$ and consider a large class of classical smoothness classes for high-dimensional approximation.

\subsection{Approximation and Smoothness Classes}
Given an approximation
tool $\Phi=(\Phi_n)_{n\geq 0}$ for the  approximation of \revY{functions from a normed space $X$},
approximation classes of $\Phi$
are sets of functions\footnote{See
\Cref{def:appclasses}.}
\revY{$A^\alpha(X,\Phi)$} for which
the error of best approximation
\begin{align}\label{eq:performance}
    E(f,\Phi_n)_{\revY{X}}:=\inf_{\varphi\in\Phi_n}\norm{f-\varphi}[\revY{X}]
\end{align} 
decays like
$n^{-\alpha}$ for $\alpha>0$.
\revQ{These classes were extensively studied 
in the context of piece-wise polynomial and wavelet approximation. One result of  approximation theory}
(see, e.g., \cite{devore_1998}) is that,
if $\Phi_n$ are piece-wise polynomials,
the classes \revY{$A^\alpha(L^p,\Phi)$} are in fact {(quasi-)} Banach spaces
and are \revQ{isomorphic} to Besov smoothness spaces.
Specifically, if the error is measured
in the $L^p(\Omega)$-norm $\norm{\cdot}[p]$,
$\Phi_n$ is the linear span of $n$ piece-wise polynomials \revQ{over a uniform partition of $\Omega$},
and $B^s_p(L^p(\Omega))$ is the Besov space\footnote{See
\Cref{sec:besov}.}
of regularity order $s$ -- as measured in the $L^p$-norm --
then, \revY{the best approximation error $E(f,\Phi_n)_{L^p}$} decays with
rate $n^{-s/D}$ if and only if $f\in B^s_p(L^p(\Omega))$.

For the case of \emph{nonlinear
approximation} \revQ{with splines} 
\revQ{(using best $n$-term approximation with hierarchical bases, or splines over adaptive non-uniform partitions)},
one has a similar characterization
but with a much weaker regularity requirement
$f\in B^s_\tau(L^\tau(\Omega))$ where
$1/\tau=s/D+1/p$. The space $B^s_\tau(L^\tau(\Omega))$ is said
to be ``on the embedding line'' since functions
in this spaces barely have enough regularity to be $p$-integrable.
This is best depicted in the DeVore diagram in \Cref{fig:DeVore} or
the embeddings in \Cref{thm:embeds}.

\begin{figure}[ht!]
        \centering
        \includegraphics[scale=.5]{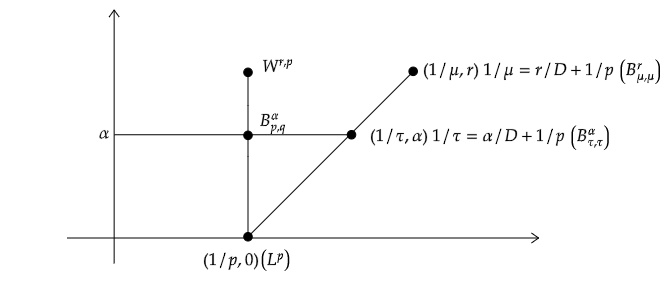}
        \caption{DeVore diagram of smoothness spaces \cite{devore_1998}.
        The Sobolev embedding line is the diagonal
        with the points $(1/\tau,\alpha)$ and $(1/\mu, r)$}.
        \label{fig:DeVore}
\end{figure}

Such results justify the claim that piece-wise polynomial
approximation, both linear and nonlinear, is mathematically
fully understood. Approximation with networks, on the other hand,
is not. This work is a contribution towards
a better understanding of these approximation tools.
We will compare TN-approximation classes with what is well-known:
piece-wise polynomial approximation classes and classical smoothness spaces.

\subsection{Main Results and Outline}
While many of the results here apply to general tree tensor networks,
unlike in the one-dimensional case, the network topology plays a crucial role
for our key embedding theorems where we require the approximation classes to be 
approximation \emph{spaces}, see also \Cref{sec:appclasses}. To that end,
we will have to restrict to linear ``TT''-like network topologies, \rev{using a particular ordering of variables}.
Our results can then be summarized in words as follows.
\begin{itemize}
    \item We show that TNs can achieve \rev{optimal or near to optimal rates of convergence for functions 
    from Sobolev or Besov spaces, with classical, mixed or anisotropic smoothness. 
    These results are obtained by  an encoding  of spline approximations using TNs.
    More precisely, we show that TNs achieve the same performance (measured by \cref{eq:performance})
    as linear approximation with splines (with uniform partition, and suitable polynomial degree),
    which is known to be optimal for some classes of Sobolev or Besov spaces.
    For a broader class of Sobolev or Besov spaces for which adaptive $n$-term approximation with
    splines (with suitable polynomial degree) is optimal, we show that TNs achieve the same performance,
    provided sparsity is exploited in the tensor representation. However, we show that using standard
    tensor formats, not exploiting sparsity, allows to achieve near to optimal approximation rates.}
    Such results are comparable to some
      expressivity results for other types of
    deep neural networks, see, e.g.,
    \cite{suzuki2018adaptivity,guhring2020,gribonval:hal-02117139,YAROTSKY2017103,
    SparseGrids,ali2020approximation}.
      \item Similar approximation results for TNs, not relying on tensorization of functions but on classical univariate approximation tools, 
      can be found in \cite{Schneider201456,griebel2021analysis,bachmayr2023approximation,griebel2023low}. 
      \rev{The interest of using tensorization of functions and the related approximation tool introduced in the present paper  is that (near to) optimal rates of convergence are achieved for a wide class of smoothness classes, without requiring to adapt the approximation tool to the smoothness of the target function. In addition, tensorization allows to exploit structures in multivariate functions that are not related to classical low-rank structures of these functions.}
    
    \item We show that for certain network topologies the approximation
    classes of TNs are
    (quasi-) Banach spaces.
 \item \rev{The above mentioned results on approximation rates imply that} Besov spaces of any order of isotropic,
    anisotropic or mixed smoothness are continuously embedded in
   the approximation
    classes of TNs.  
 \rev{However,} we also show that approximation classes of TNs themselves are not
    embedded in any Besov space.
\end{itemize}
\revY{We consider only $\Omega=[0,1)^D$.
However, we believe that} the approximation results \revY{could} be extended to
any bounded Lipschitz domain $\Omega\subset\R^D$,
see also \Cref{sec:conclusion}. \revY{This is left for future investigation.}

\subsection*{Outline}
In \Cref{sec:tensorize}, we show how $L^p(\Omega)$ can be
isometrically identified with a tensor space of any order.
We define subspaces of $L^p(\Omega)$ which will contain our approximations.
In \Cref{sec:tns}, we briefly review
tensor networks for multivariate approximation
and define our approximation tool \rev{based on TTs}.
In \Cref{sec:appclasses},
we define approximation classes of tensor networks
and show under which conditions these are (quasi-) Banach spaces.
\Cref{sec:embedd} contains our main results.
We show how spline systems can be encoded as
(or approximated by) a tensor network
and estimate the resulting complexity.
We prove approximation rates, direct and inverse embeddings.
\Cref{sec:conclusion} contains some concluding remarks, \rev{where we summarize our findings, emphasize   the role of tensorization, the influence of the ordering of variables and the choice of tensor format, and also  discuss the curse of dimensionality.}

%% file: tensorization.tex

\section{Tensorization}\label{sec:tensorize}

\revY{In \cite{ali2023approximation}, the \emph{tensorization} of one-dimensional functions is examined in detail.}
In this section, we extend the tensorization procedure to
higher dimensions. We omit details that are -- more or less --
the same as in the one-dimensional case.
Throughout this work $b$ is some integer $b=2,3,\ldots$
and $I_b:=\{0,\ldots,b-1\}$. \rev{For $Q\in \mathbb{N}$, an integer $i \in \{0,\dots,b^Q-1\}$ admits a representation in base $b$, that is $i = \sum_{k=1}^L i_k b^{Q-k} := [i_1, \dots, i_Q]_b$, with $(i_1, \dots, i_Q) \in I_b^Q$ ($Q$ bits when $b=2$). 
}

Unlike for dimension $D=1$, there are many valid approaches for tensorization
in higher dimensions. In a rather general setting,
tensorization of functions $f:\Omega\rightarrow\R$ over
a domain $\Omega\subset\R^D$ can be performed for any domain
$\Omega$ that can be encoded \rev{via some bijective transformation}\footnote{Or a diffeomorphism, depending on the intended application.}
\rev{$F:(i_1,\dots, i_Q , \bar x) \in I_b^Q\times [0,1]^D \mapsto F(i_1,\dots, i_Q , \bar x) \in \Omega$, with
$(i_1,\ldots,i_Q)\in I_b^Q$ indexing some element of a partition
of $\Omega$ with $N=b^Q$ elements, and $\bar x \in [0,1)^D $ being a local coordinate system within an element}. \revY{Even more generally, any partition with $N$ elements could be encoded with digits $(i_1, \dots , i_Q) \in I_{b_1} \times \dots \times I_{b_Q}$ given a factorization $N = b_1 \dots b_Q$. }
For the sake of a comprehensible presentation,
we will focus on $\Omega=[0,1)^D$ and a very specific tensorization
that is most relevant for \Cref{sec:appclasses,sec:embedd}.

In fact, the tensorization scheme we choose is conceptually close
to the one-dimensional case such that many properties are inherited
with similar proofs. As will become clear in \Cref{sec:appclasses}, this choice
of tensorization is required\footnote{In the sense that
other ``natural'' tensorization approaches would not lead
to linear approximation spaces, see also
\Cref{sec:p4}.} to ensure the approximation classes
we define in \Cref{def:tnappspaces} are (quasi-) Banach spaces.
However, the results presented in
this section can be extended to more general domains and tensorizations.


\subsection{The Tensorization Map}
Fix a level parameter $L\in\N_{\geq 0}$ and define the one-dimensional
conversion/encoding map $t_{b,L} : I_b^L \times [0,1) \to [0,1)$ such that 
$$t_{b,L}(i_1,\hdots,i_L , \bar x) = b^{-L} ([i_1,\hdots,i_L]_b +  \bar x), \quad [i_1,\hdots,i_L]_b := \sum_{k=1}^L i_k b^{L-k}.
$$
\rev{A univariate function $f : [0,1) \to \mathbb{R}$ can then be identified with a $(L+1)$-variate function $f \circ t_{b,L}(i_1,\hdots,i_L , \bar x)$ defined on $I_b^L \times [0,1)$.}

Then we introduce the $D$-dimensional conversion/encoding map
$\tbl:I_b^{LD}\times [0,1)^D\rightarrow [0,1)^D$
via
\begin{align*}
    \tbl(i_1^1,\ldots,i_D^1,\ldots,i_1^L,\ldots,i_D^L, \bar x_1,\ldots,\bar x_D)
    =&\\ 
    \left(t_{b,L}(i_1^1,\hdots,i_1^L,\bar x_1) , \hdots , t_{b,L}(i_D^1,\hdots,i_D^L,\bar x_D)  \right).&
\end{align*}
With this we can define a tensorization map that transforms
a $D$-variate function into a
$(L+1)D$-variate function.

\begin{definition}[Tensorization Map]\label{def:tensormap}
    We define the tensorization map $\Tbl:\R^{[0,1)^D}\rightarrow\R^{I_b^{LD}\times
    [0,1)^D}$ as
    \begin{align*}
        (\Tbl f)(i_1^1,\ldots,i_D^L,\bar x_1,\ldots,\bar x_D):= f(
        \tbl(i_1^1,\ldots,i_D^L,\bar x_1,\ldots,\bar x_D)).
    \end{align*}
\end{definition}

\begin{theorem}[Isometry]\label{th:isometry}
    Equip $[0,1)^D$ with the standard Lebesgue measure $\lambda^D$ and
    $I_b^{LD}\times [0,1)^D$ with the product measure
    $\mu_{b,L}:=\mu_b^{\otimes LD}\otimes\lambda^D$,
    where $\mu_b$ is the uniform probability measure on $I_b=\{0,\ldots,b-1\}$.
    Then, the following holds.
    \begin{enumerate}[label=(\roman*)]
        \item    The map $\Tbl$ is an isomorphism between the space of
                    (Borel) measurable functions 
                     on $[0,1)^D$  and the space of (Borel) measurable functions 
                    on $I_b^{LD}\times[0,1)^D$.
        \item    The map $\Tbl$ is an isometry between
                    $L^p([0,1)^D)$ and the tensor space
                    \begin{align*}
                        \Vbl:=\ell^p_{\mu_b}(I_b)^{\otimes LD}\otimes
                        L^p([0,1))^{\otimes D},
                    \end{align*}
                    for $0<p\leq\infty$, where $\Vbl$ is equipped with
                    the (quasi-)$L^p$-norm associated with
                    the measure $\mu_{b,L}$.
                    Moreover, $\|\cdot\|_{\Vbl}$ is a (quasi-)crossnorm, and,
                    for $1\leq p\leq\infty$, it is a reasonable crossnorm.
    \end{enumerate}
\end{theorem}

\begin{proof}
    This follows from  \cite[Proposition 2.7]{ali2023approximation} and \cite[Theorem 2.15]{ali2023approximation}
    for the case $D=1$. The results naturally extend to the case
    $D>1$. \rev{Here is a sketch of proof. To prove (i), we consider Borel sets of the form $J \times A$ with $J = \times_{k=1}^L \times_{\nu=1}^D J_{\nu,k}  \subset I_b^{LD}$ and $A $ a Borel subset of $[0,1)^D.$ They form a generating system of the Borel $\sigma$-algebra of $I_b^{LD} \times [0,1)^D$. 
    Then observing that $ t_{b,L}^D(J \times A) = \bigcup_{j \in J} A_j$ with disjoint sets $A_j$ of measure $\lambda^D(A_j) = b^{-LD} \lambda^D(A),$ we easily prove that  $\lambda^D(t_{b,L}^D(J \times A)) =  \mu_{b,L}(J \times A)$. 
    Since $T_{b,L}^D$ is a linear bijection, it preserves measurability, which yields (i). 
    The statement (ii) follows from the equality between the norm of $f$ in $L^p([0,1)^D, \lambda^D)$  and the norm of its tensorization $\boldsymbol{f}$ in  $L^p(I_b^{LD}\times [0,1)^D, \mu_{b,L})$. Indeed,  for $p<\infty$, it holds
    \begin{align*}
     \Vert f \Vert_{L^p}^p = \sum_{j \in I_b^{LD}} \int_{A_j} \vert f(x) \vert^p 
     d\lambda^D(x) = \sum_{j \in I_b^{LD}}  b^{-LD} \Vert \boldsymbol{f}(j_1^1, \dots , j_L^D , \cdot) \Vert_{L^p}^p&\\
     = \int_{I_b^{LD} \times [0,1)^D} \vert \boldsymbol{f}(j_1^1, \dots , j_L^D , \bar x) \vert^p d\mu_{b,L}(j_1^1, \dots , j^D_L , \bar x)  = \Vert \boldsymbol{f} \Vert^p_{L^p},&
    \end{align*}
    and similarly for $p=+\infty.$
    }
\end{proof}


\subsection{Finite-Dimensional Subspaces}

In this work we are concerned with approximation, and thus we
consider a subspace of
$\Vbl$ defined by
\begin{align*}
    \VblSDD:=\R^{I_b^{LD}}\otimes S^D,
\end{align*}
where $S^D\subset L^p([0,1)^D)$ is some finite-dimensional subspace.
Since $\VblSD$ can be identified with a subspace of $L^p([0,1)^D)$
through the use of the tensorization map $\Tbl$,
we set
\begin{align*}
    \vblSD:=(\Tbl)^{-1}(\VblSDD).
\end{align*}
\revQ{\begin{remark}\label{rem:coefficient_tensor}
Given a basis $\psi := (\phi_1, \dots, \phi_{M})$ of $S^D$, and element $\boldsymbol{f}$ in $\VblSDD$ admits a representation 
\begin{align*}
&\boldsymbol{f}(i_1^1, \dots, i_D^1 , \dots, i_1^L, \dots, i_D^L , \bar x_1, \dots, \bar x_D) =\\
 &\sum_{k=1}^M  \boldsymbol{C}(i_1^1, \dots, i_D^1 , \dots, i_1^{L},\dots, i_D^{L},k) \phi_{k}(\bar x_1, \dots  ,\bar x_D),
\end{align*}
which allows to identify $\boldsymbol{f}$ with the tensor $\boldsymbol{C} \in (\R^{b})^{\otimes (LD)} \otimes \R^{M}$. 
A set of integers $(i_1^1, \dots, i_D^1 , \dots, i_1^L, \dots, i_D^L)$  corresponds to a particular element $I$ of a uniform partition of $[0,1)^D$ into $b^{LD}$ hypercubes, and the entries $\{ \boldsymbol{C}(i_1^1, \dots, i_D^1 , \dots, i_1^{L},\dots, i_D^{L},k) : 1\le k \le M\}$ correspond to the coefficients of the expansion on the basis $\phi$ of the function  $f = (T_{b,L}^D)^{-1}\boldsymbol{f}$ restricted to $I$, and expressed in terms of the local variable 
$\bar x$. If $S^D$ is the space of constant functions, then $f$ is a piecewise constant function over the partition, and $\boldsymbol{f}$ is identified with a tensor $\boldsymbol{C}(i_1^1, \dots, i_D^1 , \dots, i_1^{L},\dots, i_D^{L})$ which corresponds to the value of $f$ on the element of the partition indexed by   $(i_1^1, \dots, i_D^1 , \dots, i_1^{L},\dots, i_D^{L})$.  
\end{remark}
}
As in \cite{ali2023approximation}, to ensure the approximation classes defined in \Cref{def:tnappspaces}
are
actually (quasi-) Banach spaces, it is necessary for $\vblSD$ to possess
a hierarchical structure. Thus, we require the following definition. 

\begin{definition}[Closed Under $b$-adic Dilation]
    We say $S^D$ is \emph{closed under $b$-adic dilation} if for any $f\in S^D$ and
    any $k\in\{0,\ldots,b-1\}^D$
    \begin{align*}
        f(b^{-1}(\cdot +k))\in S^D.
    \end{align*}
\end{definition}

This implies the following theorem. 

\begin{theorem}[Hierarchy of Spaces $\vblSD$]\label{thm:hierarchy}
    If $S^D\subset L^p([0,1)^D)$ is closed under $b$-adic dilation, then
    \begin{enumerate}[label=(\roman*)]
        \item    it holds
                    \begin{align*}
                        S^D=:V_{b,0,S^D}^D\subset V^D_{b,1,S^D}
                        \subset V_{b,2,S^D}^D\subset\ldots,
                    \end{align*}
        \item    the set
                    \begin{align*}
                        V^D_{b,S^D}:=\bigcup_{L=0}^\infty\vblSD
                    \end{align*}
                    is a subspace of $L^p([0,1)^D)$,
        \item    and, if $S^D$ contains the constant function one,
                    $V^{\revZ{D}}_{b,S^D}$ is dense in $L^p([0,1)^D)$ for
                    $0<p<\infty$.
    \end{enumerate}
\end{theorem}

\begin{proof}
    It follows with similar arguments as in  \cite[Proposition 2.19]{ali2023approximation}.  \rev{(i) Assume $f\in V_{b,L, S^D}^D$, that is $\boldsymbol{f} = T_{b,L}^D(f) \in \boldsymbol{V}_{b,L,S^D}^D$. Then $\boldsymbol{f}(i_1^1, \dots , i_L^D , \cdot) \in S^D$, and since  $S^D$ is closed under $b$-adic dilations, it also holds $\boldsymbol{f}(i_1^1, \dots , i_L^D , b^{-1}((i_{L+1}^1 , \dots, i_{L+1}^D) + \cdot))\in S^D$, which implies $(T_{b,L+1}^D f)(i_1^1, \dots, i^D_{L+1}, \bar x) \in S^D$, and therefore, $f\in V^D_{b,L+1 , S^D}$. The assertion (ii) follows from the fact that for two functions $f_1,f_2 \in V_{b,S^D}^D $, there exists a $L$ such that $f_1,f_2 \in V_{b,L, S^D}^D$, and a linear combination of the two functions is in  $ V_{b,L, S^D}^D \subset V_{b,S^D}^D$. To prove  (iii), we use the density  of step functions in $L^p([0,1)^D)$ for any $0<p<\infty$, and prove the density of $V_{b,S^D}^D$  in the set of step functions, by constructing a sequence of piecewise constant functions $f_L \in V_{b,L,S^D}^D$ which converges to a step function $f$ as $L\to +\infty$. 
     }
\end{proof}

In \Cref{sec:embedd}, we will employ $S^D:=(\Pm)^{\otimes D}$, where
$\Pm$ is the space of polynomials of degree $m\in\N_{\geq 0}$, restricted
to $[0,1)$. In this case we will simply write
\begin{align*}
    \Vblm:=\R^{I_b^{LD}}\otimes(\Pm)^{\otimes D},\quad
    \vblm:=(\Tbl)^{-1}(\Vblm),\quad
    V_{b,m}^D:=\bigcup_{L=0}^\infty\vblm.
\end{align*}
\revQ{\begin{remark}
The space $V_{b,L,m}^D$ corresponds to the space of (discontinuous) piecewise polynomials of partial degree $m$ over a uniform partition of $[0,1)^D$ into $N = b^{LD}$ hypercubes with measure $b^{-LD}$.  
Following Remark \ref{rem:coefficient_tensor}, given a univariate polynomial basis $\varphi_0, \dots, \varphi_{m}$ of $\mathbb{P}_{m}$, the functions $\phi_k(\bar x) =  \varphi_{j_1}(\bar x_1) \dots \varphi_{j_D}(\bar x_D)$ for $k = 1+ [j_1, \dots, j_D]_{m+1} \in \{1, \dots , M= (m+1)^D \}$ form a basis of $\mathbb{P}_m^{\otimes D}$, and the tensor $\boldsymbol{C}$ is identified with an element of $(\R^{b})^{\otimes (LD)} \otimes (\R^{m{+1}})^{\otimes D}$, which corresponds to the set of coefficients of the expansion of $f \in V_{b,L,m}^D$ in a particular basis of  (discontinuous) piecewise polynomials over a uniform partition of $[0,1)^D$. 
\end{remark}
}

%% file: tnetworks.tex

\section{Tensor Networks}\label{sec:tns}

Our approximation tool will consist of functions
in $V_{b,m}^D$ identified (through $T_{b,L}^D$) with a tensor network.
In this section, we briefly review some key notions about tensor networks
relevant to our work. 
\rev{Here, we assume that the local approximation space $S^D=S^{\otimes D}$ for some finite-dimensional space of univariate functions $S\subset L^p([0,1))$.}


\subsection{Ranks and Minimal Subspaces}
For some fixed level $L\in\N_{\geq 0}$, assume $S^D=S^{\otimes D}$
for some finite-dimensional $S$ and set
\begin{alignat*}{2}
    V_\nu&:=\R^{I_b},&&\quad\nu=1,\ldots,LD,\\
    V_\nu&:=S,&&\quad\nu=LD+1,\ldots,(L+1)D,
\end{alignat*}
i.e., $\VblSD=\bigotimes_{\nu=1}^{(L+1)D}V_\nu$. Then,
we can define the notion of (multilinear) rank
by identifying an order-$(L+1)D$ tensor from
$\VblSD$ with an order-2 tensor from
\begin{align*}
    \bs V_\beta\otimes\bs V_{\beta^c},\quad
    \bs V_\beta:=\bigotimes_{\nu\in\beta}V_\nu,\;
    \bs V_{\beta^c}:=\bigotimes_{\nu\in\beta^c}V_\nu,\quad\\
    \beta\subset\{1,\ldots,(L+1)D\},\;
    \beta^c:=\{1,\ldots,(L+1)D\}\setminus\beta.
\end{align*}
as follows.

\begin{definition}[$(\beta,L)$-Rank]\label{def:ranks}
    The $\beta$-rank of a tensor
    $\bs f\in\VblSD$ is defined as the smallest $r_\beta(\bs f)\in\N_{\geq 0}$ such that
    \begin{align*}
        \bs f=\sum_{k=1}^{r_\beta(\bs f)}\bs v_\beta^k\otimes\bs w_{\beta^c}^k,\quad
        \bs v_\beta^k\in\bs V_\beta,\;
        \bs w_{\beta^c}^k\in\bs V_{\beta^c}.
    \end{align*}
    \revQ{\begin{remark}
    Given a basis of $S$ with $\dim(S) = k$, $\boldsymbol{f}$ is identified with an algebraic  tensor $\boldsymbol{C}$ in $(\R^{b})^{\otimes (LD)} \otimes (\R^{k})^{(\otimes D)}$ (see Remark \ref{rem:coefficient_tensor}), and the rank $r_\beta(\boldsymbol{f}) = \operatorname{rank}_\beta(\boldsymbol{C}),$ that is the rank of the $\beta$-matricization (or $\beta$-unfolding) of the tensor $\boldsymbol{C}$. Thus, any low-rank representation of a functional tensor    $\boldsymbol{f}$ corresponds an equivalent  low-rank representation of the tensor $\boldsymbol{C}$ of its coefficients in a particular basis. 
    \end{remark}
}
    For a function $f\in L^p([0,1)^D)$, the $(\beta,L)$-rank
    $r_{\beta,L}(f)$ of $f$ is defined as
    \begin{align*}
        r_{\beta,L}(f):=r_\beta(\Tbl(f)).
    \end{align*}
    For $\beta=\{1,\ldots,\nu\}$, we abbreviate
    $r_{\{1,\ldots,\nu\}}(\bs f)$ and $r_{\{1,\ldots,\nu\},L}(f)$ as
    $r_{\nu}(\bs f)$ and $r_{\nu,L}(f)$, respectively.
\end{definition}

For a given $\bs f\in\VblSD$ not every combination
of ranks is possible. \revZ{In particular, we have the following result. }
\begin{lemma}[Admissible Ranks]\label{lemma:admissranks}
    Let $\bs f\in\VblSD$ and $\beta\subset\{1,\ldots,(L+1)D\}$ for some
    level $L\in\N_{\geq 0}$. Then, it holds
    $r_\beta(\bs f)=r_{\beta^c}(\bs f)$, and, for any partition
    $\beta=\gamma\cup\alpha$, it holds
    \begin{align*}
        r_\beta(\bs f)\leq r_\gamma(\bs f)r_\alpha(\bs f),
    \end{align*}
    and in particular,
    \begin{alignat*}{3}
        r_{\nu+1}(\bs f)&\leq br_\nu(\bs f)&&\quad
        \quad
        r_\nu(\bs f)\leq br_{\nu +1}(\bs f),\quad
        &&1\leq \nu\leq LD-1,\\
        r_{\nu+1}(\bs f)&\leq \dim(S)r_\nu(\bs f)&&\quad
        \quad
        r_\nu(\bs f)\leq\dim(S)r_{\nu +1}(\bs f),\quad
        &&LD\leq \nu\leq (L+1)D-1.
    \end{alignat*}
\end{lemma}

\begin{proof}
    See \cite[Lemma 2.12 and Lemma 2.23]{ali2023approximation}.
\end{proof}

A concept closely related to ranks are
\emph{minimal subspaces}.

\begin{definition}[Minimal Subspaces]
    For $\bs f\in\VblSD$ and $\beta\subset\{1,\ldots,(L+1)D\}$,
    the \emph{minimal subspace} $\Umin_\beta(\bs f)$
    of $\bs f$ is the smallest
    subspace $\bs U_\beta\subset\bs V_\beta$ such that $\bs f\in\bs U_\beta
    \otimes\bs V_{\beta^c}$, and its dimension is
    \begin{align*}
        \dim(\Umin_\beta(\bs f))=r_\beta(\bs f).
    \end{align*}
\end{definition}

For certain unfolding modes $\beta$, it is helpful to picture
$\Umin_\beta(\bs f)$ as the space of linear combinations of partial
evaluations of $\bs f$ (see also \cite[Figure 2]{ali2023approximation}).

\begin{definition}[Partial Evaluations]
    For $\bs f\in\Vbl$, $\beta\subset\{1,\ldots,LD\}$ and any
    $\nu_\beta:=(\nu_1,\ldots,\nu_{\#\beta})\in I_b^{\#\beta}$,
    we let $\bs f(\nu_\beta,\cdot)
    \in\bs V_{\beta^c}$ be a \emph{partial evaluation} of $\bs f$
    (i.e., an evaluation at $\nu_\beta$ for the variables
    of modes in $\beta$). Note that we have the
    identity (see \cite[Lemma 2.10]{ali2023approximation})
    \begin{align*}
        \Umin_{\beta^c}(\bs f)=\linspan\left\{\bs f(\nu_\beta,\cdot):\nu_\beta\in
        I_b^{\#\beta}\right\}\quad\text{and}\quad
        \dim(\Umin_\beta(\bs f))=\dim(\Umin_{\beta^c}(\bs f)).
    \end{align*}
\end{definition}

As we saw in \Cref{def:ranks}, a function $f\in L^p([0,1)^D)$ can be
associated with different levels $L\in\N_{\geq 0}$, and, in particular,
the ranks $r_{\beta,L}(f)$ may depend on the level $L$.
From \Cref{thm:hierarchy}, we know
$\vblSD\subset V_{b,L+1,S^D}\subset\ldots$. In order
to guarantee this hierarchy property, the type of \emph{level extension}
from $\vblSD$ to $V_{b,L+1,S^D}$ as implied by
\Cref{def:tensormap} is essential.

Unlike in the one-dimensional case, for $D>1$ there are many valid strategies
for increasing the representation level. E.g.,
a natural approach would be to tensorize each of the $D$ spatial dimensions
\emph{separately}, leading to $D$ level parameters $(L_1,\ldots,L_D)$.
The notion of
higher/lower level may thus be, in general, not well-defined anymore
since we only have a partial (inclusion) ordering on the set of tensor
spaces of different levels -- due to the presence of several coordinate
directions.
Moreover, even if the hierarchy property from \Cref{thm:hierarchy} is guaranteed,
for approximation with \emph{controlled} complexity we need to relate
the ranks of $r_{\beta,L}(f)$ for different $L$.
We will return to this issue in \Cref{sec:appclasses}, where we will see that
a specific choice of level extension and tensor networks guarantees
that the resulting approximation classes are (quasi-) Banach spaces.


\subsection{Tree Tensor Networks}
Let $T$ be a collection of subsets of $\{1,\ldots,(L+1)D\}$.
For a rank vector $\bs r=(r_\beta)_{\beta\in T}\in\N_{\geq 0}^{\#T}$,
we define the set of tensors in $\VblSD$ with ranks bounded by
$\bs r$ as
\begin{align*}
    \T^{T}_{\bs r}(\VblSD):=\left\{
    \bs f\in\VblSD:\;r_\beta(\bs f)\leq r_\beta,\;
    \beta\in T\right\}.
\end{align*}
We call $\T^{T}_{\bs r}(\VblSD)$ a \emph{tree-based} tensor
format if $T$ is a dimension
partition tree
(or a subset of a dimension partition tree).
In this case a tensor
$\bs f\in \T^{T}_{\bs r}(\VblSD)$ admits a parametrization with low-order
tensors and can thus be interpreted as a
\emph{tree tensor network}.
We will work with one particular type of networks.

\begin{definition}[Tensor Train (TT) Format]$\,$\\
    For $T:=\left\{\{1\},\{1,2\},\ldots,
    \{1,\ldots,(L+1)D-1\}\right\}$
    (a subset of a linear tree) and
    $\bs r=(r_\nu)_{\nu=1}^{(L+1)D-1}$,
    we call
    \begin{align*}
            \TT_{\bs r}(\VblSD):=\T^{T}_{\bs r}(\VblSD)
    \end{align*}
    the set of tensors in the \emph{tensor train (TT)} format.
\end{definition}

If $\{\varphi_k:\;1\leq k\leq\dim S\}$ is a basis of $S$,
a tensor $\bs f\in\TT_{\bs r}(\VblSD)$ admits a representation
\begin{align}\label{eq:rep}
    &\bs f(i_1^1,\ldots,i_D^L,\bar x_1,\ldots\bar x_D)
    = \\
    &\sum_{k_1=1}^{r_1}\cdots
    \sum_{k_{(L+1)D}=1}^{r_{(L+1)D}}
    \sum_{n_1,\ldots,n_D=1}^{\dim S}
     U_1(i_1^1, k_1)\cdots U_{LD}(k_{LD-1},i_D^L,k_{LD})\notag \\
     \quad \quad &U_{LD+1}(k_{LD},n_1,k_{LD+1})\varphi_{n_1}(\bar x_1)\cdots
    U_{(L+1)D}(k_{(L+1)D-1},n_D)\varphi_{n_D}(\bar x_D),\notag
\end{align}
with parameters
$U_1\in\R^{b\times r_1}$,
$U_\nu\in\R^{r_{\nu-1}\times b\times r_\nu}$ for $2\leq \nu\leq LD$,
$U_\nu\in\R^{r_{\nu-1}\times\dim S\times r_\nu}$ for $LD+1\leq \nu\leq (L+1)D-1$,
and $U_{(L+1)D}\in\R^{r_{(L+1)D-1}\times\dim S}$.
The parameters
\begin{align*}
    \bs U:=(U_1,\ldots, U_{(L+1)D})\in
    \PbS:=\R^{b\times r_1}\times\ldots\times\R^{r_{(L+1)D-1}\times\dim S},
\end{align*}
form a tree tensor network,
see \Cref{fig:tns}. Note that such tensor networks can also be identified with a recurrent
sum-product neural network,
where $r_\nu$ is the number of neurons in layer
$\nu$ (see \cite{Khrulkov2018Feb}).

\begin{figure}[ht!]
        \centering
        \input{figs/mps}
        \caption{A tensor diagram corresponding to $\bs f = \RbS(\bs U)$.}
        \label{fig:tns}
\end{figure}

With \cref{eq:rep} we can define a representation map
\begin{align*}
    \RbS:\PbS\rightarrow(\Tbl)^{-1}\left(\TT_{\bs r}(\VblSD)\right)
    \subset L^p([0,1)^D).
\end{align*}
The basis of our analysis of nonlinear approximation
in the next section will be the following set.

\begin{definition}[TT-Functions]\label{def:ttfuncs}
    For $L\in\N_{\geq 0}$, a finite-dimensional space $S$ and a
    finite TT-rank vector $\bs r=(r_\nu)_{\nu=1}^{(L+1)D-1}$,
    we define a set of TT-functions
    as
    \begin{align*}
        \PhiblS:=\left\{
        \varphi=\RbS(\bs U):\;\bs U\in\PbS\right\}.
    \end{align*}
\end{definition}

%% file: figs/mps.tex
\tikzset{every picture/.style={line width=0.75pt}} 

\tikzset{every picture/.style={line width=0.75pt}} 
\begin{tikzpicture}[x=0.75pt,y=0.75pt,yscale=-1,xscale=1]

\draw   (84.97,16.34) -- (136.13,16.34) -- (136.13,60.83) -- (84.97,60.83) -- cycle ;

\draw   (169.75,15.22) -- (220.91,15.22) -- (220.91,59.72) -- (169.75,59.72) -- cycle ;
\draw    (136.13,37.47) -- (170.23,37.47) ;
\draw    (221.64,37.47) -- (255.74,37.47) ;
\draw    (278.64,37.47) -- (312.75,37.47) ;
\draw   (313,15.22) -- (364.16,15.22) -- (364.16,59.72) -- (313,59.72) -- cycle ;
\draw    (364.16,36.41) -- (397.82,36.41) ;
\draw    (449.64,36.64) -- (473.31,36.64) ;
\draw    (111.13,60.47) -- (111.13,86.83) ;
\draw    (111.13,100) node[] {$i_1^1$} ;
\draw    (194.13,59.47) -- (194.13,85.83) ;
\draw    (194.13,100) node[] {$i_2^1$} ;

\draw    (340.13,60.47) -- (340.13,86.83) ;
\draw    (340.13,100) node[] {$i_L^D$} ;
\draw   (449.12,13.32) -- (449.18,63.89) -- (398.56,63.94) -- (398.5,13.37) -- cycle ;
\draw   (443.23,97.12) -- (443.29,147.69) -- (405.67,147.74) -- (405.61,97.17) -- cycle ;
\draw    (423.89,63.91) -- (423.93,97.63) ;
\draw    (423.98,148.44) -- (423.02,170.32) ;
\draw    (423.02,188) node[] {$\bar x_1$} ;

\draw    (501.35,36.58) -- (535.01,36.58) ;
\draw   (594.32,13.49) -- (594.38,64.06) -- (535.76,64.11) -- (535.7,13.53) -- cycle ;
\draw   (587.43,97.29) -- (587.49,147.86) -- (546.87,147.91) -- (546.81,97.34) -- cycle ;
\draw    (567.08,64.08) -- (567.12,97.79) ;
\draw    (567.18,148.61) -- (567.21,170.32) ;
\draw    (567.18,188) node[] {$\bar x_D$} ;


\draw (96.81,27.93) node [anchor=north west][inner sep=0.75pt]    {$U_{1}$};
\draw (181.6,26.82) node [anchor=north west][inner sep=0.75pt]    {$U_{2}$};
\draw (257.09,35.99) node [anchor=north west][inner sep=0.75pt]    {$\ldots $};
\draw (319.84,26.82) node [anchor=north west][inner sep=0.75pt]    {$U_{L D}$};
\draw (474.5,34.15) node [anchor=north west][inner sep=0.75pt]    {$\dotsc $};
\draw (419.95,115.11) node [anchor=north west][inner sep=0.75pt]    {$\varphi$};
\draw (400.76,27.31) node [anchor=north west][inner sep=0.75pt]    {$U_{L D + 1}$};
\draw (559.13,115.28) node [anchor=north west][inner sep=0.75pt]    {$\varphi$};
\draw (540.93,27.48) node [anchor=north west][inner sep=0.75pt]    {$U_{L D + D}$};

\end{tikzpicture}

%% file: appspaces.tex

\section{Approximation Classes}\label{sec:appclasses}
Approximation classes are sets of elements that can be
approximated with a certain rate by a pre-specified approximation tool.
It is a powerful analysis instrument for studying
approximability properties of a given tool.
\revQ{These classes were extensively studied 
for piece-wise polynomial and wavelet approximations}, see \cite{devore_1998}.
In this work and in \cite{ali2023approximation},
we apply this machinery to the study
of TN-approximation. We will see that for certain classes
one can derive strong statements about their properties.


\subsection{Basic Notions}
Let $X$ be a quasi-normed vector space and consider
subsets $\Phi_n\subset X$ for any $n\in\N_{\geq 0}$.
For the \emph{approximation tool} $\Phi:=(\Phi_n)_{n\in\N_{\geq 0}}$,
define the best approximation error
\begin{align*}
    E_n(f):=E(f, \Phi_n)_X:=\inf_{\varphi\in\Phi_n}\|f-\varphi\|_X.
\end{align*}

\begin{definition}[Approximation Classes \revQ{\cite[Section 4.1]{devore_1998}}
]\label{def:appclasses}
    For any $f\in X$, $\alpha>0$ and $0<q\leq\infty$,
    define the quasi-norm
    \begin{align*}
        \norm{f}[\Asq]:=
        \begin{cases}
            \left(\sum_{n=1}^\infty
            [n^\alpha E_{n-1}(f)]^q\frac{1}{n}\right)^{1/q},&\quad
            0<q<\infty,\\
            \sup_{n\geq 1}[n^\alpha E_{n-1}(f)],&\quad q=\infty.
        \end{cases}
    \end{align*}
    The approximation class $\Asq(X)$ of $\Phi$ is defined as
    \begin{align*}
        \Asq(X):=\Asq(X,\Phi):=\left\{f\in X:\;
        \norm{f}[\Asq]<\infty\right\}.
    \end{align*}
\end{definition}
\revQ{The class $A_{\infty}^\alpha(X)$ contains functions for which the sequence $(n^\alpha E_{n-1}(f))$ is  bounded. Belonging to the class $A_{q}^\alpha(X)$ is a slightly stronger condition since it requires that the sequence $([n^\alpha E_{n-1}(f)]^q)$ is summable with weight $1/n$}.
The following properties will be \revQ{used}
for analyzing the set
$\Asq(X)$.
\begin{enumerate}[label=(P\arabic*)]
    \item\label{P1} $0\in\Phi_n$ and $\Phi_0=\{0\}$.
    \item\label{P2} $\Phi_n\subset\Phi_{n+1}$.
    \item\label{P3} $c\Phi_n=\Phi_n$ for any $c\in\R\setminus\{0\}$.
    \item\label{P4} $\Phi_n+\Phi_n\subset\Phi_{cn}$ for some $c>0$.
    \item\label{P5} $\bigcup_{n=0}^\infty\Phi_n$ is dense in $X$.
    \item\label{P6} $\Phi_n$ is proximinal in $X$, i.e., each
    $f\in X$ has a best approximation in $\Phi_n$.
\end{enumerate}

Properties \ref{P1} -- \ref{P3} are typically easy to satisfy given an
appropriate definition of $\Phi_n$.
Property \ref{P4} is needed for $\Asq(X)$ to be a vector space and it
is in fact sufficient to derive most properties of $\Asq(X)$.
It restricts the degree of nonlinearity in $\Phi_n$ and
endows $\Asq(X)$ with a lot of structure.
Properties \ref{P5} -- \ref{P6} are mostly required for
deriving inverse embeddings for $\Asq(X)$.
Since in this work we will only show lack of inverse embeddings,
these two properties are not essential for the present exposition.


\subsection{Tensor Network Approximation Spaces}\label{sec:p4}

As our approximation tool $\Phi$,
we consider TT-functions
from \Cref{def:ttfuncs}
with different levels and ranks,
\begin{align*}
    \Phi:=\left(\PhiblS\right)_{L\in\N_{\geq 0},\bs r\in\N_{\geq 0}^{(L+1)D-1}}.
\end{align*}

\begin{remark}[Different Network Topologies]
    Note that, even though we haven't yet defined subsets of
    finite complexity $\Phi_n$, the definition of $\Phi$
    already fixes a particular type of \revQ{tree tensor networks, using a particular ordering of variables and a TT network (linear tree)}.
    In principle, we \revQ{could} define $\Phi$ as the collection of
    \rev{tree tensor networks representations} \revQ{where the tree is also a parameter}, \revZ{that includes the set  of TT networks with different orderings of the variables}.
    However, the resulting $\Phi$ would be ``too nonlinear''
    in the sense that we would not be able to define $\Phi_n$
    in such a way as to guarantee \ref{P4} and thus the resulting
    class $\Asq$ would \revQ{not be a linear space}. 
\end{remark}

As in \cite{ali2023approximation}, we introduce \rev{two} complexity measures.
\begin{definition}[Complexity Measures]
    For $\bs U\in\PbS$ we define
    \begin{enumerate}[label=(\roman*)]
        \item   \revQ{the \emph{representation} complexity}
        \begin{align*}
            \full(\bs U):=br_1+b\sum_{\nu=2}^{LD}r_{\nu-1}r_\nu+
            \dim S\sum_{\nu=LD+1}^{(L+1)D-1}r_{\nu-1}r_\nu
            +r_{(L+1)D-1}\dim S,
        \end{align*}
        \item    \revQ{the \emph{sparse} complexity}
        \begin{align*}
            \sparse(\bs U):=\sum_{\nu=1}^{(L+1)D-1}\norm{U_\nu}[\ell^0],
        \end{align*}
        where $\norm{U_\nu}[\ell^0]$ is the number of non-zero entries
        in $U_\nu$,
    \end{enumerate}
\end{definition}
\revQ{The above complexity measures are closely related to
complexity measures for feedforward neural networks: the  
representation complexity corresponds to the number of parameters of
        the neural network, while the sparse complexity corresponds to the number of nonzero parameters of the network. }
\par 
For any function $\varphi\in V_{b,S}^D$, we can define
\begin{align*}
    \full(\varphi):=\min
    \left\{
    \full(\bs U): \RbS(\bs U) = \varphi \text{ for some } L\in\N_0,
    \;\bs r\in\N_{\geq 0}^{(L+1)D-1}
    \right\},
\end{align*}
and analogously for $\sparse(\varphi)$.
\begin{definition}[Approximation Tool]\label{def:tools}
    In accordance with the previous definition, we define
    for any $n\in\N_{\geq 0}$
    \begin{align*}
        \PhiF_n&:=\left\{\varphi\in V^D_{b,S}:\;
        \full(\varphi)\leq n\right\},
    \end{align*}
    and analogously for $\PhiS_n$.
\end{definition}

For these approximation tools it is straight-forward to verify

\begin{lemma}[\ref{P1} -- \ref{P3} and \ref{P5} -- \ref{P6}]\label{lemma:ps}
    Let $S\subset L^p([0,1))$ be closed under $b$-adic dilation
    and $\dim(S)<\infty$. Then,
    \begin{enumerate}[label=(\roman*)]
        \item    \rev{The sets} $\PhiF$ and $\PhiS$ 
        both satisfy
        \ref{P1} -- \ref{P3}. For $0<p<\infty$ and if $S$ contains the constant
        function one, they also satisfy \ref{P5}.
        \item    The set $\PhiF$ 
        satisfies \ref{P6}
        for $0<p\leq\infty$, while $\PhiS$ does not.
    \end{enumerate}
\end{lemma}

\begin{proof}
    \rev{
    Properties \ref{P1} -- \ref{P3} are obvious. \ref{P5} follows from the fact that $$\bigcup_{n \in \N} \Phi_n = \bigcup_{L\in\N_{\geq 0}} \bigcup_{\bs r\in\N_{\geq 0}^{(L+1)D-1}} \PhiblS = \bigcup_{L\in\N_{\geq 0}} V_{b,L , S^{\otimes D}}^D = V_{b,S^{\otimes D}}^D,$$ and \Cref{thm:hierarchy}(iii). The proof that $\PhiF$ satisfies \ref{P6} follows the one of \cite[Lemma 3.14]{ali2023approximation}. More precisely, observe that $\PhiF_n$ is a finite union of subsets $\Phi_{b,L,S,\bs{r}}$, each being  isometrically identified to $ \TT_{\bs r}(\VblSD)
   $ through the map $T_{b,L}^D$ (\Cref{th:isometry}(ii)).  The set of tensors $\TT_{\bs r}(\VblSD)$ with bounded TT-rank being a closed subset of $\VblSD$ \cite[Lemma 8.6]{Hackbusch2012}, we deduce that $\PhiF_n$ is a closed subset of the finite-dimensional space $V_{b,L, S^{\otimes D}}^D$, hence it is proximinal in $L^p$ for any $0<p\le + \infty $. To prove that $\PhiS_n$ is not proximinal, we can follow   \cite[Example 3.15]{ali2023approximation} to construct  a sequence in  $\PhiS_n$ which converges to an element in $\PhiS_{n'}$ with $n'>n$.}
\end{proof}

Thus, it remains to check \ref{P4}. To this end, we need
to relate the ranks of a function for different representation levels
and find a particular sparse representation on a finer level
to estimate $\sparse(\cdot)$.

\begin{lemma}[Ranks $r_{\beta,L}(f)$ for Different $L$]\label{lemma:rankext}
    Let $\varphi\in\Phi_{b,L_A,S,\bs r^A}$ and assume $S$ is closed
    under $b$-adic dilation.
    Then, for any $L_B\geq L_A$,
    $\varphi\in\Phi_{b,L_B,S,\bs r^B}$ with $\bs r^B\in\N_{\geq 0}^{(L_B+1)D-1}$
    satisfying
    \begin{alignat*}{2}
        r_\nu^B&=r_\nu^A\leq\min\left\{
        b^\nu,(\dim S)^Db^{L_AD-\nu}\right\},\quad
        &&1\leq\nu\leq L_AD,\\
        r^B_\nu&\leq\min\left\{b^\nu,(\dim S)^D\right\},\quad &&
        L_AD+1\leq\nu\leq L_BD,\\
        r^B_\nu&\leq\min\left\{b^{L_BD}
        (\dim S)^{\nu-L_BD},(\dim S)^{(L_B+1)D-\nu}\right\},\quad &&
        L_BD+1\leq\nu\leq (L_B+1)D-1.
    \end{alignat*}
\end{lemma}

\begin{proof}
    Follows from \Cref{thm:hierarchy}, \Cref{lemma:admissranks} and with
    analogous arguments as
    in \cite[Lemma 2.23]{ali2023approximation}.
\end{proof}

\begin{theoremEnd}{lemma}[Sparse Complexity for Different $L$]\label{lemma:sparsity}
    Let $S$ be closed under $b$-adic dilation,
    $\varphi\in\Phi_{b, L_A, S,\bs r^A}$
    and $\bs U^A\in\mc P_{b, L_A, S,\bs r^A}$
    with $\varphi=\mc R_{b,L_A,S,\bs r^A}(\bs U^A)$.
    Then, for any $L_B>L_A$ and some
    $\bs r^B\in\N_{\geq 0}^{(L_B+1)D-1}$,
    there exists
    $\bs U^B\in\mc P_{b, L_B, S,\bs r^B}$ such that $\varphi = R_{b,L_B,S,\bs r^B}(\bs U^B)$ 
    with
    \begin{align}\label{eq:sparse}
        \sparse(\bs U^B)\leq
        \sparse(\bs U^A)
        +2b(\dim S)^{2D}(L_B-L_A)D.
    \end{align}
\end{theoremEnd}
\begin{proofEnd}
    We consider representations as in \cref{eq:rep}.
    For a representation on a given level $L_A$,
    we have to find a sparse representation on level $L_B$.
    We illustrate the proof using tensor diagram notation.
    
    For $\bs \varphi=T^D_{b,L_A}(\varphi)$, by definition, we have
    \begin{equation}\label{eq:varphiA}
        \input{figs/varphiA.tex}
    \end{equation}
    Set $L:=L_B-L_A$ and, for each $\nu=1,\ldots,D$, we tensorize the variables
    $\bar x_\nu$ as
    \begin{align*}
        \bar x_\nu=t_{b,L}^1(j_\nu^1,\ldots,j_\nu^{L},\bar y_\nu).
    \end{align*}
    Then, we can expand each $\Phi^A_\nu$ as a Tensor Train as follows
    \begin{equation}\label{eq:Phinu}
        \input{figs/Phinu.tex}
    \end{equation}
    For any $\nu=1,\ldots,D$, we have
    $\Phi^A_\nu\in\R^{r_{L_AD+\nu-1}}\otimes S\otimes \R^{r_{L_AD+\nu}}$
    and thus all ranks corresponding
    to connected edges in \cref{eq:Phinu}
    are bounded by $\dim S$ (see \cite[Lemma 2.7]{ali2023approximation}).
    
    Next, inserting \cref{eq:Phinu} into \cref{eq:varphiA}, we obtain
    for $\bs \varphi=T^D_{b,L_B}(\varphi)$
    \begin{equation}\label{eq:phiB}
        \input{figs/phiB.tex}
    \end{equation}
    Note that the representation from \cref{eq:phiB} is not the
    same as \cref{eq:rep}. To obtain the latter we have to
    rearrange the part of the tree in \cref{eq:rep} after $U^A_{L_AD}$
    such that cores corresponding to the same tensorization level
    are grouped together.
    This can be achieved by inserting tensor products with
    identities ``in-between''. We first illustrate this with a simple example.
    
    Suppose we have the following three tensors
    \begin{equation*}
        \input{figs/ABC.tex}
    \end{equation*}
    and we want to contract edges numbered 1 and 2 with
    each other such that we obtain
    \begin{equation}\label{eq:contract}
        \input{figs/contract.tex}
    \end{equation}
    Now, suppose we want to obtain the same result but
    we are restricted to contracting all edges of $A$ with all edges of
    $B$ first, and then all edges of the resulting tensor with all edges
    of $C$. This can be accomplished by the following modification
    of $B$
    \begin{equation*}
        \input{figs/modB.tex}
    \end{equation*}
    where $\id$ is the identity mapping on the space
    corresponding to index 1.
    With this modification we can contract $A$ with $\bar B$
    \begin{equation*}
        \input{figs/contract1.tex}
    \end{equation*}
    and then the result with $C$
    \begin{equation*}
        \input{figs/contract2.tex}
    \end{equation*}
    such that overall we obtain the same result as in \cref{eq:contract}.
    Note that the number of non-zero entries
    of the modified tensor $\bar B$ is equal to the number of non-zero
    entries of $B$ times the dimension of the space
    corresponding to index 1.
    
    Now we apply this to our original problem in
    \cref{eq:phiB} as follows. The first $L_AD$ cores remain unchanged,
    i.e., $U_\nu^B:=U_\nu^A$ for $\nu=1,\ldots, L_AD$.
    Next, to keep track of the edges from
    \cref{eq:phiB}, we use the labels
    $\mathrm{d}(\nu\leftrightarrow\nu+1)$ to indicate edges
    connecting cores of different spatial dimensions and
    $\mathrm{lvl}_\nu(j\leftrightarrow j+1)$ to indicate edges connecting
    cores of different levels within the spatial variable $\nu$.
    That is, we use the following labeling
    \begin{equation*}
        \input{figs/labeled.tex}
    \end{equation*}
    where we do not label the free edges, since we will not
    perform any modification w.r.t.\ to the corresponding indices.
    Then, the first new $D$ cores are defined as
    
    \begin{equation*}
        \input{figs/reordering1.tex}
    \end{equation*}
    
    where the input variable corresponds to the free unlabeled edge.

    The second level of the new $D$ cores is defined as
    
    \begin{equation*}
        \input{figs/reordering2.tex}
    \end{equation*}

    Finally, the last $D$ cores are defined as
    
    \begin{equation*}
        \input{figs/reordering3.tex}
    \end{equation*}
    
    Counting the number of non-zero entries of $\bs U^B$,
    we see that the number of non-zero entries
    in the cores upto $L_AD$ is bounded by
    $\sparse(\bs U^A)$, while the remainder
    is bounded by the constant
    $2b(\dim S)^{2D}(L_B-L_A)D$.
\end{proofEnd}

\begin{lemma}[\ref{P4}]\label{lemma:p4}
    If $S$ is closed under $b$-adic dilation,
    then the  sets $\PhiF$ and $\PhiS$ satisfy \ref{P4}
    with $c=\mc O([\dim S]^{2D})$.
\end{lemma}

\begin{proof}
    (i) First, we show the statement for $\PhiF$. 
    Let $\varphi_A,\varphi_B\in\PhiF_n$ such that
    $\varphi_A\in\Phi_{b, L_A, S,\bs r^A}$,
    $\varphi_B\in\Phi_{b, L_B, S,\bs r^B}$ and w.l.o.g.\
    $L_A\leq L_B$. Then, there exists
    $\bs U\in\mc P_{b,L_B,S,\bs r}$ with
    $\varphi_A+\varphi_B=\mc R_{b,L_B,S,\bs r} (\bs U)$ and,
    by invoking \Cref{lemma:admissranks,lemma:rankext},
    \begin{alignat*}{1}
        \full(\bs U)&\leq b(r_1^A+r_1^B)+b\sum_{\nu=2}^{L_AD}
        (r^A_{\nu-1}+r^B_{\nu-1})(r^A_\nu+r^B_\nu)\\
        & +b\sum_{\nu=L_AD+1}^{L_BD}([\dim S]^D+r^B_{\nu-1})
        ([\dim S]^D+r_\nu^B)\\
        & +\dim S\sum_{\nu=L_BD+1}^{(L_B+1)D-1}([\dim S]^{(L_B+1)D-\nu+1}
        +r^B_{\nu-1})
        ([\dim S]^{(L_B+1)D-\nu}
        +r^B_{\nu})\\
        & +(\dim S+r^B_{(L_B+1)D-1})\dim S\\
        &\leq br_1^A+b\sum_{\nu=2}^{L_AD}r^A_{\nu-1}r^A_\nu
        + b\sum_{\nu=2}^{L_AD}r^A_{\nu-1}r^B_\nu\\
        &+br_1^B+b\sum_{\nu=2}^{L_AD}r^B_{\nu-1}r^B_\nu
        + b \sum_{\nu=2}^{L_AD}r^B_{\nu-1}r^A_\nu+
        b\sum_{\nu=L_AD+1}^{L_BD}r^B_{\nu-1}r^B_\nu \\
        &
        + b(\dim S)^D\sum_{\nu=L_AD+1}^{L_BD}r^B_{\nu-1}+r^B_\nu+\dim S\sum_{\nu=L_BD+1}^{(L_B+1)D-1}r^B_{\nu-1}r^B_\nu\\
        &
        +
        (\dim S)^{D+1}\sum_{\nu=L_BD+1}^{(L_B+1)D-1}r^B_\nu
        +(\dim S)^D\sum_{\nu=L_BD+1}^{(L_B+1)D-1}r^B_{\nu-1}+\dim Sr^B_{(L_B+1)D-1}
        \\ &
        +b(\dim S)^{2D}(L_B-L_A)D+(\dim S)^{D+1}(D-1)
        +(\dim S)^2\\
        &\leq
        \underbrace{(2+2b+3[\dim S]^D+
        [\dim S]^{D+1}+[\dim S]^D+
        [\dim S]^{2D})}_{=:c}n\\
        &+
        (\dim S)^{D+1}(D-1)+(\dim S)^2,
    \end{alignat*}
    where we used $L_B-L_A\leq L_B\leq\full(\bs U^B)\leq n$ and the fact that 
    $$\sum_\nu r_{\nu-1}^A r_\nu^B \le (\sum_\nu (r_{\nu-1}^A)^2 )^{1/2}(\sum_\nu (r_{\nu}^B)^2 )^{1/2} \le
    (\sum_\nu b r_{\nu-1}^A r_\nu^A )^{1/2}(\sum_\nu  b r_{\nu-1}^B r_{\nu}^B )^{1/2} \le n ,$$
    and similarly for bounding $\sum_\nu r_{\nu-1}^B r_\nu^A$ by $n$.
    This shows the statement for $\PhiF$.
    
    (ii) For two functions
    $\varphi_A\in\Phi_{b,L,S,\bs r^A}$, $\varphi_B\in\Phi_{b,L,S,\bs r^B}$,
    it is not difficult to
    see that for $\varphi_A+\varphi_B=
    \mc R_{b,L,S,\bs r}(\bs U)$,
    it holds $\sparse(\bs U)\leq\sparse(\bs U^A)+\sparse(\bs U^B)$:
    we refer to \cite[Lemma 3.12]{ali2023approximation} for more details.
    Thus, together with \Cref{lemma:sparsity},
    this shows \ref{P4} for $\Phi^{\mathcal{S}}$ with
    \begin{align*}
        c:=2(1+(\dim S)^{2D}).
    \end{align*}
\end{proof}

\begin{remark}[Index ordering and \ref{P4}]
    A key part of the proof in \Cref{lemma:p4} is \Cref{lemma:rankext},
    i.e., when increasing the resolution level $L\in\N$ the ranks either remain
    the same or are bounded by a constant independent of $L$.
    We stress that for this to hold the tree network structure was essential:
    choosing a different tree network can lead to
    $\Phi_n$ that violate \ref{P4}.
    
    For instance, consider a natural re-ordering of the tree considered in
    \Cref{eq:rep} as follows.
    We keep a uniform parameter $L\in\N$ for the resolution
    depth in each dimension: i.e., we do not consider
    anisotropic approximation.
    We re-label the   variables as follows
    \begin{align*}
        (\underbrace{i_1^1,\ldots,i_D^L,\bar x_1}_{\text{1st
        coordinate}},\ldots,\underbrace{i_D^1,\ldots,\bar x_D}_{\text{$D$-th coordinate}}).
    \end{align*}
    Consider a TT-network with the index order
    corresponding to the above labeling. Then, unlike in \Cref{lemma:rankext},
    for any $l=0,1,\ldots$, we can only bound the ranks as
    $r_{L+l}\leq(b^L\dim S)^{(D-1)}$
    or $r_{L+l}\leq r_L\min\{b^l,(\dim S)^{D-1}\}$, i.e.,
    this bound depends either on $L$, or $L$ and $l$.
    In this situation, as was demonstrated in \cite[Proposition 3.4]{ali2023approximation},
    one can add two functions $\varphi_A,\varphi_B\in\Phi_n$,
    where one is high resolution and of low-rank
    and vice versa. The sum is both of high resolution and high-rank
    and, thus, we can at best only ensure
    $\Phi_n+\Phi_n\subset\Phi_{cn^2}$.
    In conclusion, \ref{P4} can be guaranteed only for specific
    structures of tree networks.
\end{remark}

\begin{remark}[The constant $c$ in \ref{P4}]
    The constant $c$ in $\Phi_n+\Phi_n\subset\Phi_{cn}$ depends
    on $D$ as $(\dim S)^{2D}$:
    it is exponential in $D$ if $\dim S>1$.
    The structure of our tree network
    implies the function space we consider for \revZ{functions} of the last $D$
    variables $(\bar x_1,\ldots,\bar x_D)$
    is $S^{\otimes D}$
    and hence the curse of dimensionality for $\dim S>1$. 

    On one hand,
    the chosen tree structure (ordering of variables)
    ensures \ref{P4} is satisfied.
    On the other hand,
    this comes at the expense of an
    unfavorable tree choice for separating variables
    corresponding to different coordinate directions, such that,
    in the worst case, the representation complexity of a sum
    of two approximands will be much larger than the
    respective individual representation complexities.
    
    If we were to combine approximations that are adaptive
    in the ranks separating different coordinate directions,
    the ranks separating different discretization levels in each coordinate
    direction \textbf{and} depth of discretization --
    this would result in a highly non-linear approximation tool that
	avoids the aforementioned exponential constant scaling    
    but \revQ{violates the property \ref{P4} and thus,  the corresponding approximation spaces $A^\alpha_q$ would not be linear spaces.}  
    
    Note however that for using for $S$ the space of constant functions ($\dim(S)=1$) results in a constant $c$ independent of $D$.   
\end{remark}

\begin{definition}[TT-Approximation Classes]\label{def:tnappspaces}
    For $0<p,q\leq\infty$ and $\alpha>0$, we consider the following approximation classes
    \begin{align*}
        \Fsq&:=\Asq(L^p, \PhiF),\\
        \Ssq&:=\Asq(L^p, \PhiS)
        .
    \end{align*}
\end{definition}

We obtain the first main result of this work.

\begin{theorem}[TT-Approximation Spaces]\label{thm:appspaces}
    Let $S$ be closed under $b$-adic dilation, $0<p,q\leq\infty$
    and $\alpha>0$.
    Then, the classes $\Fsq$ and $\Ssq$ 
    are (quasi-)Banach spaces
    satisfying the continuous embeddings
    \begin{align}\label{eq:appspacesembeds}
        \Fsq\hookrightarrow\Ssq
        \hookrightarrow
        \mc F^{\alpha/2}_{q}(L^p).
    \end{align}
\end{theorem}

\begin{proof}
    The first statement follows from
    \Cref{lemma:ps,lemma:p4}, see also \cite[Theorem 3.17]{ali2023approximation}
    for $D=1$.
    \Cref{eq:appspacesembeds} follows by similar arguments as in
    \cite[Theorem 3.19]{ali2023approximation}. \rev{More precisely, following the proof of  \cite[Proposition 3.18]{ali2023approximation} (easily extended to the case $D>1$), one can prove that for any $\varphi \in V_{b,S^{\otimes D}}^D$, 
    $$
    \sparse(\varphi) \le \full(\varphi) \le d + c \, \sparse(\varphi)^2,
    $$
    for some constants $d,c.$ The first inequality is obvious. The second one 
    can be obtained  by considering an intermediate complexity measure  $\neuron(\varphi)$ given by the sum of the TT-ranks of $\varphi$, and noting that $ \full(\varphi) \le d + c \, \neuron(\varphi)^2$ and $\neuron(\varphi) \le \sparse(\varphi)$.  This implies $$\PhiF_n \subset \PhiS_n \subset \PhiF_{d + c \, n^2},$$ and we deduce the embeddings from the definition of approximation spaces.}
\end{proof}

%% file: figs/varphiA.tex
\tikzset{every picture/.style={line width=0.75pt}} 
\begin{tikzpicture}[x=0.75pt,y=0.75pt,yscale=-1,xscale=1]

\draw   (88.97,90.34) -- (140.13,90.34) -- (140.13,134.83) -- (88.97,134.83) -- cycle ;

\draw   (173.75,89.22) -- (224.91,89.22) -- (224.91,133.72) -- (173.75,133.72) -- cycle ;
\draw    (140.13,111.47) -- (174.23,111.47) ;
\draw    (225.64,111.47) -- (259.74,111.47) ;
\draw    (282.64,111.47) -- (316.75,111.47) ;
\draw   (317,89.22) -- (368.16,89.22) -- (368.16,133.72) -- (317,133.72) -- cycle ;
\draw    (368.16,109.25) -- (402.26,109.25) ;
\draw   (402.51,87) -- (453.67,87) -- (453.67,131.5) -- (402.51,131.5) -- cycle ;
\draw   (544.51,87) -- (595.67,87) -- (595.67,131.5) -- (544.51,131.5) -- cycle ;
\draw    (453.64,109.47) -- (487.74,109.47) ;
\draw    (510.64,109.47) -- (544.75,109.47) ;
\draw    (115.13,134.47) -- (115.13,160.83) ;
\draw    (198.13,133.47) -- (198.13,159.83) ;
\draw    (344.13,134.47) -- (344.13,160.83) ;
\draw    (429.13,131.47) -- (429.13,157.83) ;
\draw    (570.13,131.47) -- (570.13,157.83) ;

\draw (50.02,102.76) node [anchor=north west][inner sep=0.75pt]    {$\boldsymbol{\varphi }^A =$};
\draw (100.81,101.93) node [anchor=north west][inner sep=0.75pt]    {$U^{A}_{1}$};
\draw (185.6,100.82) node [anchor=north west][inner sep=0.75pt]    {$U^{A}_{2}$};
\draw (260,115) node [anchor=north west][inner sep=0.75pt]    {$\ldots $};
\draw (323.84,100.82) node [anchor=north west][inner sep=0.75pt]    {$U^{A}_{L_{A} D}$};
\draw (414.36,98.59) node [anchor=north west][inner sep=0.75pt]    {$\Phi ^{A}_{1}$};
\draw (554.36,98.59) node [anchor=north west][inner sep=0.75pt]    {$\Phi ^{A}_{D}$};
\draw (490, 115) node [anchor=north west][inner sep=0.75pt]    {$\ldots $};
\draw (108,166.4) node [anchor=north west][inner sep=0.75pt]    {$i^{1}_{1}$};
\draw (191,165.4) node [anchor=north west][inner sep=0.75pt]    {$i^{1}_{2}$};
\draw (337,166.4) node [anchor=north west][inner sep=0.75pt]    {$i^{L_{A}}_{D}$};
\draw (422,163.4) node [anchor=north west][inner sep=0.75pt]    {$\overline{x}_{1}$};
\draw (563,163.4) node [anchor=north west][inner sep=0.75pt]    {$\overline{x}_{D}$};

\end{tikzpicture}

%% file: figs/Phinu.tex
\tikzset{every picture/.style={line width=0.75pt}} 
\begin{tikzpicture}[x=0.75pt,y=0.75pt,yscale=-1,xscale=1]

\draw   (275.97,90.34) -- (327.13,90.34) -- (327.13,134.83) -- (275.97,134.83) -- cycle ;
\draw   (360.75,89.22) -- (411.91,89.22) -- (411.91,133.72) -- (360.75,133.72) -- cycle ;
\draw    (327.13,111.47) -- (361.23,111.47) ;
\draw    (412.64,111.47) -- (446.74,111.47) ;
\draw   (505.51,89) -- (556.67,89) -- (556.67,133.5) -- (505.51,133.5) -- cycle ;
\draw    (471.64,111.47) -- (505.75,111.47) ;
\draw    (300.13,134.47) -- (300.13,160.83) ;
\draw    (385.13,133.47) -- (385.13,159.83) ;
\draw    (531.13,133.47) -- (531.13,159.83) ;
\draw    (241.13,111.47) -- (275.23,111.47) ;
\draw    (300.13,63.47) -- (300.13,89.83) ;

\draw (95.02,101.76) node [anchor=north west][inner sep=0.75pt]    {$\Phi _{\nu } =$};
\draw (287.81,100.93) node [anchor=north west][inner sep=0.75pt]    {$W^{1}_{\nu}$};
\draw (372.6,100.82) node [anchor=north west][inner sep=0.75pt]    {$W^{2}_{\nu}$};
\draw (515.36,100.59) node [anchor=north west][inner sep=0.75pt]    {$W^{L+1}_{\nu}$};
\draw (450,115) node [anchor=north west][inner sep=0.75pt]    {$\ldots $};
\draw (291,166.4) node [anchor=north west][inner sep=0.75pt]    {$j^{1}_{\nu }$};
\draw (378,165.4) node [anchor=north west][inner sep=0.75pt]    {$j^{1}_{\nu }$};
\draw (524,165.4) node [anchor=north west][inner sep=0.75pt]    {$\overline{y}_{\nu }$};
\draw (170,99.4) node [anchor=north west][inner sep=0.75pt]    {$k_{L_{A} D+\nu -1}$};
\draw (277,30.4) node [anchor=north west][inner sep=0.75pt]    {$k_{L_{A} D+\nu }$};

\end{tikzpicture}

%% file: figs/phiB.tex
\tikzset{every picture/.style={line width=0.75pt}} 
\begin{tikzpicture}[x=0.75pt,y=0.75pt,yscale=-1,xscale=1]

\draw   (84.97,16.34) -- (136.13,16.34) -- (136.13,60.83) -- (84.97,60.83) -- cycle ;

\draw   (169.75,15.22) -- (220.91,15.22) -- (220.91,59.72) -- (169.75,59.72) -- cycle ;
\draw    (136.13,37.47) -- (170.23,37.47) ;
\draw    (221.64,37.47) -- (255.74,37.47) ;
\draw    (278.64,37.47) -- (312.75,37.47) ;
\draw   (313,15.22) -- (364.16,15.22) -- (364.16,59.72) -- (313,59.72) -- cycle ;
\draw    (364.16,36.41) -- (397.82,36.41) ;
\draw    (439.64,36.64) -- (473.31,36.64) ;
\draw    (111.13,60.47) -- (111.13,86.83) ;
\draw    (194.13,59.47) -- (194.13,85.83) ;
\draw    (340.13,60.47) -- (340.13,86.83) ;
\draw   (439.12,13.32) -- (439.18,63.89) -- (398.56,63.94) -- (398.5,13.37) -- cycle ;
\draw   (440.23,97.12) -- (440.29,147.69) -- (399.67,147.74) -- (399.61,97.17) -- cycle ;
\draw    (419.89,63.91) -- (419.93,97.63) ;
\draw    (419.98,148.44) -- (420.02,182.15) ;
\draw   (454.47,240.19) -- (454.53,290.77) -- (400.04,290.83) -- (399.98,240.26) -- cycle ;
\draw    (420.05,206.76) -- (420.09,240.48) ;
\draw    (399.87,121.27) -- (375.81,121.3) ;
\draw    (400.03,265.59) -- (375.97,265.62) ;
\draw    (398.56,63.94) -- (373.84,79) ;
\draw    (501.35,36.58) -- (535.01,36.58) ;
\draw   (576.32,13.49) -- (576.38,64.06) -- (535.76,64.11) -- (535.7,13.53) -- cycle ;
\draw   (577.43,97.29) -- (577.49,147.86) -- (536.87,147.91) -- (536.81,97.34) -- cycle ;
\draw    (557.08,64.08) -- (557.12,97.79) ;
\draw    (557.18,148.61) -- (557.21,182.32) ;
\draw   (591.47,240.36) -- (591.53,290.93) -- (537.23,291) -- (537.17,240.43) -- cycle ;
\draw    (557.24,206.93) -- (557.28,240.65) ;
\draw    (537.06,121.44) -- (513,121.47) ;
\draw    (537.23,265.75) -- (513.17,265.79) ;
\draw    (535.76,64.11) -- (511.03,79.17) ;

\draw (22.02,118.76) node [anchor=north west][inner sep=0.75pt]    {$\boldsymbol{\varphi } =$};
\draw (96.81,27.93) node [anchor=north west][inner sep=0.75pt]    {$U^{A}_{1}$};
\draw (181.6,26.82) node [anchor=north west][inner sep=0.75pt]    {$U^{A}_{2}$};
\draw (257.09,35.99) node [anchor=north west][inner sep=0.75pt]    {$\ldots $};
\draw (319.84,26.82) node [anchor=north west][inner sep=0.75pt]    {$U^{A}_{L_{A} D}$};
\draw (474.5,34.15) node [anchor=north west][inner sep=0.75pt]    {$\dotsc $};
\draw (426.6,182.34) node [anchor=north west][inner sep=0.75pt]  [rotate=-89.93]  {$\dotsc $};
\draw (407.73,252.72) node [anchor=north west][inner sep=0.75pt]    {$W^{L+1}_{1}$};
\draw (407.95,111.11) node [anchor=north west][inner sep=0.75pt]    {$W^{2}_{1}$};
\draw (407.76,27.31) node [anchor=north west][inner sep=0.75pt]    {$W^{1}_{1}$};
\draw (563.8,182.51) node [anchor=north west][inner sep=0.75pt]  [rotate=-89.93]  {$\dotsc $};
\draw (544.92,252.89) node [anchor=north west][inner sep=0.75pt]    {$W^{L+1}_{D}$};
\draw (545.13,111.28) node [anchor=north west][inner sep=0.75pt]    {$W^{2}_{D}$};
\draw (544.93,27.48) node [anchor=north west][inner sep=0.75pt]    {$W^{1}_{D}$};

\end{tikzpicture}

%% file: figs/ABC.tex
\tikzset{every picture/.style={line width=0.75pt}} 
\begin{tikzpicture}[x=0.75pt,y=0.75pt,yscale=-1,xscale=1]

\draw   (88.97,90.34) -- (140.13,90.34) -- (140.13,134.83) -- (88.97,134.83) -- cycle ;
\draw    (140.13,111.47) -- (174.23,111.47) ;
\draw    (115.13,134.47) -- (115.13,160.83) ;
\draw   (287.97,90.34) -- (339.13,90.34) -- (339.13,134.83) -- (287.97,134.83) -- cycle ;
\draw    (253.13,110.47) -- (287.23,110.47) ;
\draw   (423.97,87.34) -- (475.13,87.34) -- (475.13,131.83) -- (423.97,131.83) -- cycle ;
\draw    (450.13,131.47) -- (450.13,157.83) ;

\draw (106.81,103.93) node [anchor=north west][inner sep=0.75pt]    {$A$};
\draw (110,163.4) node [anchor=north west][inner sep=0.75pt]    {$1$};
\draw (176,102.4) node [anchor=north west][inner sep=0.75pt]    {$2$};
\draw (305.81,103.93) node [anchor=north west][inner sep=0.75pt]    {$B$};
\draw (243,101.4) node [anchor=north west][inner sep=0.75pt]    {$2$};
\draw (441.81,100.93) node [anchor=north west][inner sep=0.75pt]    {$C$};
\draw (445,159.4) node [anchor=north west][inner sep=0.75pt]    {$1$};

\end{tikzpicture}

%% file: figs/contract.tex
\tikzset{every picture/.style={line width=0.75pt}} 
\begin{tikzpicture}[x=0.75pt,y=0.75pt,yscale=-1,xscale=1]

\draw   (88.97,90.34) -- (140.13,90.34) -- (140.13,134.83) -- (88.97,134.83) -- cycle ;
\draw    (140.13,111.47) -- (174.23,111.47) ;
\draw    (115.13,134.47) -- (115.13,160.83) ;
\draw   (221.97,90.34) -- (273.13,90.34) -- (273.13,134.83) -- (221.97,134.83) -- cycle ;
\draw    (187.13,110.47) -- (221.23,110.47) ;
\draw   (89.97,207.34) -- (141.13,207.34) -- (141.13,251.83) -- (89.97,251.83) -- cycle ;
\draw    (115.13,180.47) -- (115.13,206.83) ;

\draw (106.81,103.93) node [anchor=north west][inner sep=0.75pt]    {$A$};
\draw (110,163.4) node [anchor=north west][inner sep=0.75pt]    {$1$};
\draw (176,102.4) node [anchor=north west][inner sep=0.75pt]    {$2$};
\draw (239.81,103.93) node [anchor=north west][inner sep=0.75pt]    {$B$};
\draw (107.81,220.93) node [anchor=north west][inner sep=0.75pt]    {$C$};

\end{tikzpicture}

%% file: figs/modB.tex
\tikzset{every picture/.style={line width=0.75pt}} 
\begin{tikzpicture}[x=0.75pt,y=0.75pt,yscale=-1,xscale=1]

\draw   (306.97,110.34) -- (358.13,110.34) -- (358.13,154.83) -- (306.97,154.83) -- cycle ;
\draw   (150.97,106.34) -- (202.13,106.34) -- (202.13,150.83) -- (150.97,150.83) -- cycle ;
\draw    (178.63,73.02) -- (178.69,107.13) ;
\draw    (202.13,129.47) -- (236.23,129.47) ;
\draw    (177.13,183.83) -- (177.13,150.83) ;
\draw   (379.97,110.34) -- (431.13,110.34) -- (431.13,154.83) -- (379.97,154.83) -- cycle ;
\draw    (332.13,187.83) -- (332.13,154.83) ;
\draw    (432.13,133.47) -- (466.23,133.47) ;
\draw    (405.13,110.83) -- (405.13,77.83) ;
\draw  [dash pattern={on 4.5pt off 4.5pt}] (302,105) -- (436.67,105) -- (436.67,160.83) -- (302,160.83) -- cycle ;

\draw (324.81,123.93) node [anchor=north west][inner sep=0.75pt]    {$B$};
\draw (168.81,119.93) node [anchor=north west][inner sep=0.75pt]    {$\overline{B}$};
\draw (172,186.4) node [anchor=north west][inner sep=0.75pt]    {$2$};
\draw (172.09,55.29) node [anchor=north west][inner sep=0.75pt]    {$1$};
\draw (240,120.4) node [anchor=north west][inner sep=0.75pt]    {$1$};
\draw (257,123.4) node [anchor=north west][inner sep=0.75pt]    {$:=$};
\draw (361.81,123.93) node [anchor=north west][inner sep=0.75pt]    {$\otimes $};
\draw (397.81,123.93) node [anchor=north west][inner sep=0.75pt]    {$\mathbb{I}$};
\draw (327,191.4) node [anchor=north west][inner sep=0.75pt]    {$2$};
\draw (470,124.4) node [anchor=north west][inner sep=0.75pt]    {$1$};
\draw (400.09,59.29) node [anchor=north west][inner sep=0.75pt]    {$1$};

\end{tikzpicture}

%% file: figs/contract1.tex
\tikzset{every picture/.style={line width=0.75pt}} 
\begin{tikzpicture}[x=0.75pt,y=0.75pt,yscale=-1,xscale=1]

\draw   (230.97,105.34) -- (282.13,105.34) -- (282.13,149.83) -- (230.97,149.83) -- cycle ;
\draw   (303.97,105.34) -- (355.13,105.34) -- (355.13,149.83) -- (303.97,149.83) -- cycle ;
\draw    (256.13,182.83) -- (256.13,149.83) ;
\draw    (356.13,128.47) -- (390.23,128.47) ;
\draw    (329.13,105.83) -- (329.13,72.83) ;
\draw  [dash pattern={on 4.5pt off 4.5pt}] (226,100) -- (360.67,100) -- (360.67,155.83) -- (226,155.83) -- cycle ;
\draw   (369.97,170.34) -- (421.13,170.34) -- (421.13,214.83) -- (369.97,214.83) -- cycle ;
\draw    (396.13,170.83) -- (396.13,137.83) ;
\draw    (265.5,192.47) -- (368.23,192.47) ;

\draw (248.81,118.93) node [anchor=north west][inner sep=0.75pt]    {$B$};
\draw (285.81,118.93) node [anchor=north west][inner sep=0.75pt]    {$\otimes $};
\draw (321.81,118.93) node [anchor=north west][inner sep=0.75pt]    {$\mathbb{I}$};
\draw (252,182.4) node [anchor=north west][inner sep=0.75pt]    {$2$};
\draw (391,119.4) node [anchor=north west][inner sep=0.75pt]    {$1$};
\draw (324.09,54.29) node [anchor=north west][inner sep=0.75pt]    {$1$};
\draw (387.81,183.93) node [anchor=north west][inner sep=0.75pt]    {$A$};

\end{tikzpicture}

%% file: figs/contract2.tex
\tikzset{every picture/.style={line width=0.75pt}} 
\begin{tikzpicture}[x=0.75pt,y=0.75pt,yscale=-1,xscale=1]

\draw   (203.97,149.34) -- (255.13,149.34) -- (255.13,193.83) -- (203.97,193.83) -- cycle ;
\draw   (276.97,149.34) -- (328.13,149.34) -- (328.13,193.83) -- (276.97,193.83) -- cycle ;
\draw    (229.13,226.83) -- (229.13,193.83) ;
\draw    (329.13,172.47) -- (363.23,172.47) ;
\draw    (302.13,149.83) -- (302.13,116.83) ;
\draw  [dash pattern={on 4.5pt off 4.5pt}] (199,144) -- (333.67,144) -- (333.67,199.83) -- (199,199.83) -- cycle ;
\draw   (342.97,214.34) -- (394.13,214.34) -- (394.13,258.83) -- (342.97,258.83) -- cycle ;
\draw    (369.13,214.83) -- (369.13,181.83) ;
\draw    (238.5,236.47) -- (341.23,236.47) ;
\draw   (276.97,19.34) -- (328.13,19.34) -- (328.13,63.83) -- (276.97,63.83) -- cycle ;
\draw    (302.13,96.83) -- (302.13,63.83) ;

\draw (221.81,162.93) node [anchor=north west][inner sep=0.75pt]    {$B$};
\draw (258.81,162.93) node [anchor=north west][inner sep=0.75pt]    {$\otimes $};
\draw (294.81,162.93) node [anchor=north west][inner sep=0.75pt]    {$\mathbb{I}$};
\draw (225,226.4) node [anchor=north west][inner sep=0.75pt]    {$2$};
\draw (364,163.4) node [anchor=north west][inner sep=0.75pt]    {$1$};
\draw (297.09,98.29) node [anchor=north west][inner sep=0.75pt]    {$1$};
\draw (360.81,227.93) node [anchor=north west][inner sep=0.75pt]    {$A$};
\draw (294.81,32.93) node [anchor=north west][inner sep=0.75pt]    {$C$};

\end{tikzpicture}

%% file: figs/labeled.tex
\tikzset{every picture/.style={line width=0.75pt}} 
\begin{tikzpicture}[x=0.75pt,y=0.75pt,yscale=-1,xscale=1]

\draw    (62.64,37.47) -- (96.75,37.47) ;
\draw   (97,15.22) -- (148.16,15.22) -- (148.16,59.72) -- (97,59.72) -- cycle ;
\draw    (148.16,36.41) -- (231.67,36.41) ;
\draw    (272.64,36.64) -- (335.67,36.64) ;
\draw    (124.13,60.47) -- (124.13,86.83) ;
\draw   (272.12,13.32) -- (272.18,63.89) -- (231.56,63.94) -- (231.5,13.37) -- cycle ;
\draw   (271.23,127.12) -- (271.29,177.69) -- (230.67,177.74) -- (230.61,127.17) -- cycle ;
\draw    (252.89,63.91) -- (252.89,126.5) ;
\draw    (250.98,178.44) -- (250.98,221.5) ;
\draw   (285.47,294.19) -- (285.53,344.77) -- (231.04,344.83) -- (230.98,294.26) -- cycle ;
\draw    (251.09,243.5) -- (251.09,294.48) ;
\draw    (230.87,151.27) -- (206.81,151.3) ;
\draw    (231.03,319.59) -- (206.97,319.62) ;
\draw    (231.56,63.94) -- (206.84,79) ;
\draw    (422.67,35.58) -- (520.01,35.58) ;
\draw   (561.32,12.49) -- (561.38,63.06) -- (520.76,63.11) -- (520.7,12.53) -- cycle ;
\draw    (542.08,63.08) -- (542.08,125.5) ;
\draw    (520.76,63.11) -- (496.03,78.17) ;
\draw   (561.23,126.29) -- (561.29,176.86) -- (520.67,176.91) -- (520.61,126.34) -- cycle ;
\draw    (540.98,177.61) -- (540.98,220.67) ;
\draw   (575.47,293.36) -- (575.53,343.93) -- (521.04,344) -- (520.98,293.43) -- cycle ;
\draw    (541.09,242.67) -- (541.09,293.65) ;
\draw    (520.87,150.44) -- (496.81,150.47) ;
\draw    (521.03,318.75) -- (496.97,318.79) ;

\draw (38.09,30.99) node [anchor=north west][inner sep=0.75pt]    {$\ldots $};
\draw (103.84,26.82) node [anchor=north west][inner sep=0.75pt]    {$U^{A}_{L_{A} D}$};
\draw (365.5,36.15) node [anchor=north west][inner sep=0.75pt]    {$\ldots $};
\draw (254.6,220.34) node [anchor=north west][inner sep=0.75pt]  [rotate=-89.93]  {$\ldots $};
\draw (238.73,306.72) node [anchor=north west][inner sep=0.75pt]    {$W^{L+1}_{1}$};
\draw (238.95,141.11) node [anchor=north west][inner sep=0.75pt]    {$W^{2}_{1}$};
\draw (240.76,27.31) node [anchor=north west][inner sep=0.75pt]    {$W^{1}_{1}$};
\draw (529.93,26.48) node [anchor=north west][inner sep=0.75pt]    {$W^{1}_{D}$};
\draw (160,20.4) node [anchor=north west][inner sep=0.75pt]  [font=\footnotesize]  {$\mathrm{d}( 0\leftrightarrow 1)$};
\draw (544.6,219.51) node [anchor=north west][inner sep=0.75pt]  [rotate=-89.93]  {$\ldots $};
\draw (528.73,305.89) node [anchor=north west][inner sep=0.75pt]    {$W^{L+1}_{D}$};
\draw (528.95,140.28) node [anchor=north west][inner sep=0.75pt]    {$W^{2}_{D}$};
\draw (282,19.4) node [anchor=north west][inner sep=0.75pt]  [font=\footnotesize]  {$\mathrm{d}( 1\leftrightarrow 2)$};
\draw (418,18.4) node [anchor=north west][inner sep=0.75pt]  [font=\footnotesize]  {$\mathrm{d}( D-1\leftrightarrow D)$};
\draw (258,87.4) node [anchor=north west][inner sep=0.75pt]  [font=\footnotesize]  {$\mathrm{lvl}_{1}( 1\leftrightarrow 2)$};
\draw (257,195.4) node [anchor=north west][inner sep=0.75pt]  [font=\footnotesize]  {$\mathrm{lvl}_{1}( 2\leftrightarrow 3)$};
\draw (258,266.4) node [anchor=north west][inner sep=0.75pt]  [font=\footnotesize]  {$\mathrm{lvl}_{1}( L\leftrightarrow L+1)$};
\draw (550,87.4) node [anchor=north west][inner sep=0.75pt]  [font=\footnotesize]  {$\mathrm{lvl}_{D}( 1\leftrightarrow 2)$};
\draw (550,190.4) node [anchor=north west][inner sep=0.75pt]  [font=\footnotesize]  {$\mathrm{lvl}_{D}( 2\leftrightarrow 3)$};
\draw (548,263.4) node [anchor=north west][inner sep=0.75pt]  [font=\footnotesize]  {$\mathrm{lvl}_{D}( L\leftrightarrow L+1)$};

\end{tikzpicture}

%% file: figs/reordering1.tex
\tikzset{every picture/.style={line width=0.75pt}} 
\begin{tikzpicture}[x=0.75pt,y=0.75pt,yscale=-1,xscale=1]

\draw   (310,37) -- (373.67,37) -- (373.67,83.83) -- (310,83.83) -- cycle ;
\draw    (266.67,60) -- (310.67,60) ;
\draw    (374,46) -- (414.67,46) ;
\draw    (374,70) -- (413.67,70) ;
\draw    (343,8) -- (343,36.83) ;
\draw   (311,158) -- (374.67,158) -- (374.67,204.83) -- (311,204.83) -- cycle ;
\draw    (344,129) -- (344,157.83) ;
\draw    (266.67,183) -- (310.67,183) ;
\draw   (316,226) -- (370.67,226) -- (370.67,263.83) -- (316,263.83) -- cycle ;
\draw    (375,169) -- (415.67,169) ;
\draw    (375,193) -- (414.67,193) ;
\draw    (270.67,244) -- (314.67,244) ;
\draw    (371,243) -- (410.67,243) ;
\draw  [dash pattern={on 4.5pt off 4.5pt}] (305,153) -- (384.67,153) -- (384.67,272.83) -- (305,272.83) -- cycle ;
\draw   (311,329) -- (374.67,329) -- (374.67,375.83) -- (311,375.83) -- cycle ;
\draw    (344,300) -- (344,328.83) ;
\draw    (266.67,354) -- (310.67,354) ;
\draw   (316,397) -- (370.67,397) -- (370.67,434.83) -- (316,434.83) -- cycle ;
\draw    (375,340) -- (415.67,340) ;
\draw    (375,364) -- (414.67,364) ;
\draw    (270.67,415) -- (314.67,415) ;
\draw    (371,414) -- (410.67,414) ;
\draw   (317,461) -- (371.67,461) -- (371.67,498.83) -- (317,498.83) -- cycle ;
\draw    (271.67,479) -- (315.67,479) ;
\draw    (372,478) -- (411.67,478) ;
\draw  [dash pattern={on 4.5pt off 4.5pt}] (304,322) -- (389.67,322) -- (389.67,513.83) -- (304,513.83) -- cycle ;

\draw (66,51.4) node [anchor=north west][inner sep=0.75pt]    {$U^{B}_{L_{A} D+1} :=$};
\draw (329,49.4) node [anchor=north west][inner sep=0.75pt]    {$W^{1}_{1}$};
\draw (422,35.4) node [anchor=north west][inner sep=0.75pt]    {$\mathrm{d}( 1\leftrightarrow 2)$};
\draw (421,61.4) node [anchor=north west][inner sep=0.75pt]    {$\mathrm{lvl}_{1}( 1\leftrightarrow 2)$};
\draw (199,50.4) node [anchor=north west][inner sep=0.75pt]    {$\mathrm{d}( 0\leftrightarrow 1)$};
\draw (67,198.4) node [anchor=north west][inner sep=0.75pt]    {$U^{B}_{L_{A} D+2} :=$};
\draw (330,170.4) node [anchor=north west][inner sep=0.75pt]    {$W^{1}_{2}$};
\draw (199,173.4) node [anchor=north west][inner sep=0.75pt]    {$\mathrm{d}( 1\leftrightarrow 2)$};
\draw (335,208.4) node [anchor=north west][inner sep=0.75pt]    {$\otimes $};
\draw (339,235.4) node [anchor=north west][inner sep=0.75pt]    {$\mathbb{I}$};
\draw (423,158.4) node [anchor=north west][inner sep=0.75pt]    {$\mathrm{d}( 2\leftrightarrow 3)$};
\draw (422,184.4) node [anchor=north west][inner sep=0.75pt]    {$\mathrm{lvl}_{2}( 1\leftrightarrow 2)$};
\draw (187,234.4) node [anchor=north west][inner sep=0.75pt]    {$\mathrm{lvl}_{1}( 1\leftrightarrow 2)$};
\draw (418,234.4) node [anchor=north west][inner sep=0.75pt]    {$\mathrm{lvl}_{1}( 1\leftrightarrow 2)$};
\draw (67,369.4) node [anchor=north west][inner sep=0.75pt]    {$U^{B}_{L_{A} D+3} :=$};
\draw (330,341.4) node [anchor=north west][inner sep=0.75pt]    {$W^{1}_{3}$};
\draw (199,344.4) node [anchor=north west][inner sep=0.75pt]    {$\mathrm{d}( 2\leftrightarrow 3)$};
\draw (335,379.4) node [anchor=north west][inner sep=0.75pt]    {$\otimes $};
\draw (339,406.4) node [anchor=north west][inner sep=0.75pt]    {$\mathbb{I}$};
\draw (423,329.4) node [anchor=north west][inner sep=0.75pt]    {$\mathrm{d}( 3\leftrightarrow 4)$};
\draw (422,355.4) node [anchor=north west][inner sep=0.75pt]    {$\mathrm{lvl}_{3}( 1\leftrightarrow 2)$};
\draw (187,405.4) node [anchor=north west][inner sep=0.75pt]    {$\mathrm{lvl}_{1}( 1\leftrightarrow 2)$};
\draw (418,405.4) node [anchor=north west][inner sep=0.75pt]    {$\mathrm{lvl}_{1}( 1\leftrightarrow 2)$};
\draw (336,440.4) node [anchor=north west][inner sep=0.75pt]    {$\otimes $};
\draw (340,470.4) node [anchor=north west][inner sep=0.75pt]    {$\mathbb{I}$};
\draw (188,469.4) node [anchor=north west][inner sep=0.75pt]    {$\mathrm{lvl}_{2}( 1\leftrightarrow 2)$};
\draw (419,469.4) node [anchor=north west][inner sep=0.75pt]    {$\mathrm{lvl}_{2}( 1\leftrightarrow 2)$};
\draw (339.1,546.33) node [anchor=north west][inner sep=0.75pt]  [rotate=-90]  {$\dotsc $};

\end{tikzpicture}

%% file: figs/reordering2.tex
\tikzset{every picture/.style={line width=0.75pt}} 
\begin{tikzpicture}[x=0.75pt,y=0.75pt,yscale=-1,xscale=1]

\draw   (322,42) -- (385.67,42) -- (385.67,88.83) -- (322,88.83) -- cycle ;
\draw    (355,13) -- (355,41.83) ;
\draw    (277.67,67) -- (321.67,67) ;
\draw   (327,110) -- (381.67,110) -- (381.67,147.83) -- (327,147.83) -- cycle ;
\draw    (386,65) -- (426.67,65) ;
\draw    (281.67,128) -- (325.67,128) ;
\draw    (382,127) -- (421.67,127) ;
\draw   (328,210) -- (382.67,210) -- (382.67,247.83) -- (328,247.83) -- cycle ;
\draw    (282.67,228) -- (326.67,228) ;
\draw    (383,227) -- (422.67,227) ;
\draw  [dash pattern={on 4.5pt off 4.5pt}] (315,35) -- (399.67,35) -- (399.67,259) -- (315,259) -- cycle ;

\draw   (324,322) -- (387.67,322) -- (387.67,368.83) -- (324,368.83) -- cycle ;
\draw    (357,293) -- (357,321.83) ;
\draw    (279.67,347) -- (323.67,347) ;
\draw   (329,390) -- (383.67,390) -- (383.67,427.83) -- (329,427.83) -- cycle ;
\draw    (388,345) -- (428.67,345) ;
\draw    (283.67,408) -- (327.67,408) ;
\draw    (384,407) -- (423.67,407) ;
\draw   (330,490) -- (384.67,490) -- (384.67,527.83) -- (330,527.83) -- cycle ;
\draw    (284.67,508) -- (328.67,508) ;
\draw    (385,507) -- (424.67,507) ;
\draw  [dash pattern={on 4.5pt off 4.5pt}] (317,315) -- (401.67,315) -- (401.67,539) -- (317,539) -- cycle ;

\draw (48,132.4) node [anchor=north west][inner sep=0.75pt]    {$U^{B}_{( L_{A} +1) D+1} :=$};
\draw (341,54.4) node [anchor=north west][inner sep=0.75pt]    {$W^{2}_{1}$};
\draw (343,89.4) node [anchor=north west][inner sep=0.75pt]    {$\otimes $};
\draw (350,119.4) node [anchor=north west][inner sep=0.75pt]    {$\mathbb{I}$};
\draw (182,57.4) node [anchor=north west][inner sep=0.75pt]    {$\mathrm{lvl}_{1}( 1\leftrightarrow 2)$};
\draw (351,219.4) node [anchor=north west][inner sep=0.75pt]    {$\mathbb{I}$};
\draw (434,55.4) node [anchor=north west][inner sep=0.75pt]    {$\mathrm{lvl}_{1}( 2\leftrightarrow 3)$};
\draw (193,117.4) node [anchor=north west][inner sep=0.75pt]    {$\mathrm{lvl}_{2}( 1\leftrightarrow 2)$};
\draw (433,116.4) node [anchor=north west][inner sep=0.75pt]    {$\mathrm{lvl}_{2}( 1\leftrightarrow 2)$};
\draw (179,218.4) node [anchor=north west][inner sep=0.75pt]    {$\mathrm{lvl}_{D}( 1\leftrightarrow 2)$};
\draw (345,151.4) node [anchor=north west][inner sep=0.75pt]    {$\otimes $};
\draw (355.1,166.33) node [anchor=north west][inner sep=0.75pt]  [rotate=-90]  {$\dotsc $};
\draw (344,189.4) node [anchor=north west][inner sep=0.75pt]    {$\otimes $};
\draw (433,217.4) node [anchor=north west][inner sep=0.75pt]    {$\mathrm{lvl}_{D}( 1\leftrightarrow 2)$};
\draw (50,412.4) node [anchor=north west][inner sep=0.75pt]    {$U^{B}_{( L_{A} +1) D+2} :=$};
\draw (343,334.4) node [anchor=north west][inner sep=0.75pt]    {$W^{2}_{2}$};
\draw (345,369.4) node [anchor=north west][inner sep=0.75pt]    {$\otimes $};
\draw (352,399.4) node [anchor=north west][inner sep=0.75pt]    {$\mathbb{I}$};
\draw (184,337.4) node [anchor=north west][inner sep=0.75pt]    {$\mathrm{lvl}_{2}( 1\leftrightarrow 2)$};
\draw (353,499.4) node [anchor=north west][inner sep=0.75pt]    {$\mathbb{I}$};
\draw (436,335.4) node [anchor=north west][inner sep=0.75pt]    {$\mathrm{lvl}_{2}( 2\leftrightarrow 3)$};
\draw (195,397.4) node [anchor=north west][inner sep=0.75pt]    {$\mathrm{lvl}_{1}( 2\leftrightarrow 3)$};
\draw (435,396.4) node [anchor=north west][inner sep=0.75pt]    {$\mathrm{lvl}_{2}( 2\leftrightarrow 3)$};
\draw (181,498.4) node [anchor=north west][inner sep=0.75pt]    {$\mathrm{lvl}_{D}( 1\leftrightarrow 2)$};
\draw (435,497.4) node [anchor=north west][inner sep=0.75pt]    {$\mathrm{lvl}_{D}( 1\leftrightarrow 2)$};
\draw (346,469.4) node [anchor=north west][inner sep=0.75pt]    {$\otimes $};
\draw (357.1,446.33) node [anchor=north west][inner sep=0.75pt]  [rotate=-90]  {$\dotsc $};
\draw (347,431.4) node [anchor=north west][inner sep=0.75pt]    {$\otimes $};
\draw (340.1,575.33) node [anchor=north west][inner sep=0.75pt]  [rotate=-90]  {$\dotsc $};

\end{tikzpicture}

%% file: figs/reordering3.tex
\tikzset{every picture/.style={line width=0.75pt}} 
\begin{tikzpicture}[x=0.75pt,y=0.75pt,yscale=-1,xscale=1]

\draw   (322,42) -- (385.67,42) -- (385.67,88.83) -- (322,88.83) -- cycle ;
\draw    (355,13) -- (355,41.83) ;
\draw    (277.67,67) -- (321.67,67) ;
\draw   (327,110) -- (381.67,110) -- (381.67,147.83) -- (327,147.83) -- cycle ;
\draw    (281.67,128) -- (325.67,128) ;
\draw    (382,127) -- (421.67,127) ;
\draw   (328,210) -- (382.67,210) -- (382.67,247.83) -- (328,247.83) -- cycle ;
\draw    (282.67,228) -- (326.67,228) ;
\draw    (383,227) -- (422.67,227) ;
\draw  [dash pattern={on 4.5pt off 4.5pt}] (315,35) -- (399.67,35) -- (399.67,259) -- (315,259) -- cycle ;

\draw   (324,322) -- (387.67,322) -- (387.67,368.83) -- (324,368.83) -- cycle ;
\draw    (357,293) -- (357,321.83) ;
\draw    (279.67,347) -- (323.67,347) ;
\draw   (329,390) -- (383.67,390) -- (383.67,427.83) -- (329,427.83) -- cycle ;
\draw    (283.67,408) -- (327.67,408) ;
\draw    (384,407) -- (423.67,407) ;
\draw   (330,490) -- (384.67,490) -- (384.67,527.83) -- (330,527.83) -- cycle ;
\draw    (284.67,508) -- (328.67,508) ;
\draw    (385,507) -- (424.67,507) ;
\draw  [dash pattern={on 4.5pt off 4.5pt}] (317,315) -- (401.67,315) -- (401.67,539) -- (317,539) -- cycle ;
\draw   (316,669) -- (379.67,669) -- (379.67,715.83) -- (316,715.83) -- cycle ;
\draw    (349,640) -- (349,668.83) ;
\draw    (271.67,694) -- (315.67,694) ;

\draw (48,132.4) node [anchor=north west][inner sep=0.75pt]    {$U^{B}_{L_{B} D+1} :=$};
\draw (341,54.4) node [anchor=north west][inner sep=0.75pt]    {$W^{L+1}_{1}$};
\draw (343,89.4) node [anchor=north west][inner sep=0.75pt]    {$\otimes $};
\draw (350,119.4) node [anchor=north west][inner sep=0.75pt]    {$\mathbb{I}$};
\draw (164,57.4) node [anchor=north west][inner sep=0.75pt]    {$\mathrm{lvl}_{1}( L\leftrightarrow L+1)$};
\draw (351,219.4) node [anchor=north west][inner sep=0.75pt]    {$\mathbb{I}$};
\draw (168,117.4) node [anchor=north west][inner sep=0.75pt]    {$\mathrm{lvl}_{2}( L\leftrightarrow L+1)$};
\draw (161,218.4) node [anchor=north west][inner sep=0.75pt]    {$\mathrm{lvl}_{D}( L\leftrightarrow L+1)$};
\draw (345,151.4) node [anchor=north west][inner sep=0.75pt]    {$\otimes $};
\draw (354.1,166.33) node [anchor=north west][inner sep=0.75pt]  [rotate=-90]  {$\dotsc $};
\draw (344,189.4) node [anchor=north west][inner sep=0.75pt]    {$\otimes $};
\draw (50,412.4) node [anchor=north west][inner sep=0.75pt]    {$U^{B}_{L_{B} D+2} :=$};
\draw (343,334.4) node [anchor=north west][inner sep=0.75pt]    {$W^{L+1}_{2}$};
\draw (345,369.4) node [anchor=north west][inner sep=0.75pt]    {$\otimes $};
\draw (352,399.4) node [anchor=north west][inner sep=0.75pt]    {$\mathbb{I}$};
\draw (353,499.4) node [anchor=north west][inner sep=0.75pt]    {$\mathbb{I}$};
\draw (346,469.4) node [anchor=north west][inner sep=0.75pt]    {$\otimes $};
\draw (356.1,446.33) node [anchor=north west][inner sep=0.75pt]  [rotate=-90]  {$\dotsc $};
\draw (347,431.4) node [anchor=north west][inner sep=0.75pt]    {$\otimes $};
\draw (346.1,575.33) node [anchor=north west][inner sep=0.75pt]  [rotate=-90]  {$\dotsc $};
\draw (431,117.4) node [anchor=north west][inner sep=0.75pt]    {$\mathrm{lvl}_{2}( L\leftrightarrow L+1)$};
\draw (433,218.4) node [anchor=north west][inner sep=0.75pt]    {$\mathrm{lvl}_{D}( L\leftrightarrow L+1)$};
\draw (162,338.4) node [anchor=north west][inner sep=0.75pt]    {$\mathrm{lvl}_{2}( L\leftrightarrow L+1)$};
\draw (167,399.4) node [anchor=north west][inner sep=0.75pt]    {$\mathrm{lvl}_{3}( L\leftrightarrow L+1)$};
\draw (432,397.4) node [anchor=north west][inner sep=0.75pt]    {$\mathrm{lvl}_{3}( L\leftrightarrow L+1)$};
\draw (163,498.4) node [anchor=north west][inner sep=0.75pt]    {$\mathrm{lvl}_{D}( L\leftrightarrow L+1)$};
\draw (434,497.4) node [anchor=north west][inner sep=0.75pt]    {$\mathrm{lvl}_{D}( L\leftrightarrow L+1)$};
\draw (43,683.4) node [anchor=north west][inner sep=0.75pt]    {$U^{B}_{( L_{B} +1) D} :=$};
\draw (335,681.4) node [anchor=north west][inner sep=0.75pt]    {$W^{L+1}_{D}$};
\draw (152,685.4) node [anchor=north west][inner sep=0.75pt]    {$\mathrm{lvl}_{D}( L\leftrightarrow L+1)$};

\end{tikzpicture}

%% file: embeddings.tex

\section{Embeddings of Smoothness Classes}\label{sec:embedd}
In this section, \revY{we review spline systems} and we show how a linear combination of splines can be encoded as a
TT and estimate the resulting complexity. This will lead us to approximation rates for \revY{functions from various Besov spaces, which are characterized by spline systems (see \Cref{sec:review}).}  \revY{More precisely, we consider 
the isotropic Besov spaces $B^{s_\I}_q(L^p(\Omega))$
corresponding to the vertical line in \Cref{fig:DeVore}}, \revY{the more general anisotropic Besov spaces $AB^{\bs{\alpha}}_q(L^p(\Omega))$ with smoothness parameters $\bs{\alpha} = (\alpha_1, \dots , \alpha_D) \in \R^D_{>0} $, and 
   the Besov spaces of mixed dominating smoothness $MB^{s_\M}_q(L^p(\Omega))$. For an anisotropic smoothness parameter $\bs{\alpha}=(\alpha_1, \dots , \alpha_D)$, we let $\underline \alpha := \min_\nu \alpha_\nu$, $\bar \alpha := \max_\nu \alpha_\nu$, and we define the 
    aggregated smoothness $$s_\A:=s_\A(\bs{\alpha}):=
    D\left(\alpha_1^{-1}+\ldots+\alpha_D^{-1}\right)^{-1}.$$
    }\revY{Note that $B^{s_\I}_q(L^p(\Omega))$ is a particular instance of $AB^{\bs{\alpha}}_q(L^p(\Omega))$ with $\bs{\alpha} = (s_\I, \dots ,s_\I)$ and $s_\A = \bar \alpha = \underline{\alpha} = s_\I$.}
  
Interestingly, the complexity of encoding
classical approximation tools differs for linear and nonlinear
spline approximation, where in the nonlinear case the sparsity
of the tensor cores will play an important role.
We conclude by showing that the approximation class of TTs is not embedded in any
Besov space.
\par
\revQ{Throughout this section, the notation $n\lesssim k$ (resp.  $n \gtrsim k$) means that there exists a constant  $C$ (resp. $c$) such that $n \le C k$ (resp. $n \ge c k$). We write  $n \sim k$ when   both  $n\lesssim k$ and $n \gtrsim k$ hold. The dependence of these constants on additional parameters will be stated explicitly when  relevant.}

\subsection{Spline systems}\label{sec:spline-systems}
It is well known that optimal approximation
for the Besov spaces \revY{defined in \Cref{sec:besov}} can be achieved by
either systems of dilated splines
or wavelets, the latter being numerically advantageous
as wavelets form a stable multiscale basis for Besov spaces.
Since in this work we use classical approximation tools
only as an intermediate proof vehicle, we will focus
on the theoretically simpler case of spline systems.

\begin{definition}[Dilated Splines]\label{def:affsys}
    Let $\phi_{\bar m}:\R\rightarrow\R$ 
    be the cardinal B-spline of
    polynomial degree $\bar m\in\N_{\geq 0}$.
    Let $l\in\N_{\geq 0}$ and $j\in\{-\bar m,\ldots,b^l-1\}$.
    Then, the univariate spline $\varphi_{l,j}$ dilated $l$-times and shifted
    by $j$ (and normalized in $L^p$) is defined as
    \begin{align*}
        \varphi_{l,j}(x):=b^{l/p}\phi_{\bar m}(b^lx-j).
    \end{align*}
    The $D$-dimensional multivariate spline is defined by taking tensor
    products as
    \begin{align*}
        \varphi_{(l_1,\ldots,l_D), (j_1,\ldots,j_D)}^D:=
        \bigotimes_{\nu=1}^D\varphi_{l_\nu, j_\nu}.
    \end{align*}
\end{definition}

The dilated splines defined above have $D$ resolution levels $l_1,\ldots,l_D$.
We will define three types of spline systems where each is perfectly
suited for approximation in the three types of Besov spaces
introduced \revY{in \Cref{sec:review}}. Intuitively, it is clear
that for isotropic Besov spaces the effective resolution level
should be the same in all coordinate directions;
for anisotropic, the effective resolution levels in each coordinate
direction should
vary according to the smoothness parameters
$\bs\alpha=(\alpha_1,\ldots,\alpha_D)$;
for mixed, the effective resolution level
should be the sum of the unidirectional resolution levels
$\sum_\nu l_\nu$, as is the case in hyperbolic cross approximations.

\begin{definition}[Multi-Dimensional Spline Systems]\leavevmode
    \begin{enumerate}[label=(\roman*)]
        \item We define the \emph{isotropic} index set
        \begin{align*}
            \mc J_\I:=\bigcup_{l=0}^\infty
            \{(l,\ldots,l)\}\times
            \{-\bar m,\ldots,b^l-1\}^D.
        \end{align*}
        Correspondingly, we define the isotropic spline system
        \begin{align*}
            \Phi_\I:=\left\{\varphi_{\lambda_\I}:\;
            \lambda_\I\in\mc J_\I\right\}.
        \end{align*}
        We use the shorthand notation
        $|\lambda_\I|:=|(l,\ldots,l,j_1,\ldots,j_D)|:=l$ \revZ{for the level of the index $\lambda_\I$.}        
        
        \item For the anisotropic case, let $\bs\alpha=(\alpha_1,\ldots,\alpha_D)$
        denote the smoothness multi-index.
        Set $\alpha'_\nu:=\underline{\alpha}(1/\alpha_\nu)$.
        For a level parameter $l\in\N_{\geq 0}$,
        we define the $\nu$-th level as
        $l_\nu(l):=\lfloor l\alpha'_\nu\rfloor$.
        With this
        we define the \emph{anisotropic} index set as
        \begin{align*}
            \mc J_\A:=\bigcup_{l=0}^\infty
            \{(l_1(l),\ldots,l_D(l))\}\times
            \{-\bar m,\ldots,b^{l_1(l)}-1\}\times\ldots\times
            \{-\bar m,\ldots,b^{l_D(l)}-1\}.
        \end{align*}
        The \emph{anisotropic} spline system is defined accordingly as
        \begin{align*}
            \Phi_\A:=\left\{\varphi_{\lambda_\A}:\;\lambda_\A\in\mc J_\A\right\},
        \end{align*}
        and we again use the shorthand notation
        $|\lambda_\A|:=|(l_1(l),\ldots,l_D(l),j_1,\ldots,j_D)|:=l$. \revZ{An index $\lambda_\A$ such that  $|\lambda_\A| = l$ has different levels $l_1(l),\ldots,l_D(l)$ in the different dimensions, which depend  on the anisotropic smoothness parameters $\boldsymbol{\alpha}$}. \rev{In the special case $\boldsymbol{\alpha} = (\alpha,\dots, \alpha)$, $\mc J_\A = \mc J_\I$ and  $\Phi_\A = \Phi_\I$.}

        \item Finally, for the mixed case we define the
        index set as        
        \begin{align*}
            \mc J_\M:=\bigcup_{(l_1,\ldots,l_D)\in(\N_{\geq 0)^D}}
            \{(l_1,\ldots,l_D)\}\times
            \{-\bar m,\ldots,b^{l_1}-1\}\times\ldots\times
            \{-\bar m,\ldots,b^{l_D}-1\},
        \end{align*}
         and the spline system as
        \begin{align*}
            \Phi_\M:=\left\{\varphi_{\lambda_\M}:\;\lambda_\M\in\mc J_\M\right\}.
        \end{align*}
        We use the shorthand notation
        $|\lambda_\M|_1:=|(l_1,\ldots,l_D)|_1:=l_1+\ldots+l_D$.
    \end{enumerate}
\end{definition}

Having defined spline systems,
we need to specify how an element of a Besov space can be
decomposed in a given system.
For reasons of numerical stability,
one would typically decompose functions
in a wavelet system that forms a stable basis.
Our results would remain the same for this approach,
however, with a tighter restriction on the integrability
parameter $p$, since
for $0<p\leq 1$ one would have to replace
the $L^p$-space with a Hardy space.
Thus, to avoid unnecessary technicalities,
we stick to the spline characterization.

Assume
\begin{align}\label{eq:mbar_high_enough}
    \min\{\bar m+1,\bar m+1/p\}\geq s_\I,s_\M, \bar\alpha,
\end{align}
depending on the space in question.
We illustrate the decomposition procedure for the isotropic case,
all others being analogous.
Introduce a uniform partition $D_l$ of
$[0,1)^D$ into $b^{lD}$ elements of measure $b^{-lD}$.
For each element $K\in D_l$, introduce a near-best $L^p$ polynomial
approximation $P_K$ on $K$ of degree $\bar m$ of $f$.
Let $S_l(f)$ be defined piecewise such that $S_l(f)=P_K$ on $K$.
Finally, introduce a quasi-interpolant $ Q_l(S_l(f))\in\linspan\Phi_\I$
as defined in \cite[Section 4]{devore1988interpolation}.
The final level $l$ approximation in
the spline system $\Phi_\I$ is
$T_l(f):=Q_l(S_l(f))$ for any $f\in L^p([0,1)^D)$.
We then decompose $f$ as
\begin{align}\label{eq:decomp}
    f=\sum_{l=0}^\infty T_{l}(f)-T_{l-1}(f)=
    \sum_{l=0}^\infty\sum_{|\lambda_\I|=l}d_{\lambda_\I,p}(f)\varphi_{\lambda_\I},
\end{align}
with the convention $T_{-1}(f)=0$.
I.e., the coefficients $d_{\lambda_\I,p}(f)$
are the spline coefficients\footnote{Normalized in $L^p$, since
we normalized $\varphi_{\lambda_\I}$ in $L^p$.} of the level differences
$T_{l}(f)-T_{l-1}(f)$ of quasi-interpolants of near-best polynomial
approximations of $f$.
For comparison, in a wavelet basis these coefficients
would simply be the $L^2$-inner products with wavelets
on the corresponding level (up to re-scaling).

The construction for the anisotropic case is analogous,
adjusting the index sets $\mc J$ and projections accordingly, \revZ{see \cite{Leisner}}. For the mixed
case, one requires projections onto the hyperbolic cross
consisting of splines corresponding
to the index set
$\Lambda_L:=\{\lambda_\M\in\mc J_\M:\;
|\lambda_\M|_1=L\}$,
i.e., these can be constructed via adding
details for all multilevels $l\in \N^D$ that satisfy
$\sum_\nu l_\nu\leq L$.
In literature such constructions are typically
performed via wavelets, which form a stable basis for the detail
spaces where $T_l(f)-T_{l-1}(f)$ lives.
This would not change our analysis and we stick to
splines to avoid unnecessary technicalities.
For details we refer to \cite{devore1988interpolation,Leisner,Hansen2012,NonCompact}.

\subsection{Encoding Splines}
An $n$-term sum of dilated splines is of the form
$$\varphi_n:=\sum_{\lambda\in\Lambda}  \revQ{c_{\lambda}}\varphi_\lambda, \quad \#\Lambda\leq n, $$
where 
$\lambda=(l_1,\ldots,l_D,j_1,\ldots,j_D)$ for some resolution levels 
$(l_1,\ldots,l_D)\in(\N_{\geq 0})^D$,  and 
$$
\varphi_\lambda :=\varphi_{l_1,j_1}\otimes\ldots\otimes
        \varphi_{l_D,j_D},  \quad \varphi_{l,j}(x) := b^{l/p} \phi_{\bar m}(b^l x - j),
$$
with $\phi_{\bar m}$ the cardinal B-spline of polynomial degree $\bar m$, \revY{and $\varphi_{l,j}(x)$ a univariate spline normalized in $L^p$}. 
We refer to Section \ref{sec:spline-systems} for a detailed introduction of spline systems and related approximation results.  
To encode $\varphi_n$ as a TT in $V_{b,m}^D$, we proceed as follows.
\begin{enumerate}
    \item    For $\bar m\leq m$, we represent the cardinal B-spline $\phi_{\bar m}:
    \R\rightarrow\R$ as a TT
    in $V_{b,m}^1$ and estimate
    the resulting complexity.
    For $\bar m>m$, we approximate $\phi_{\bar m}$
    by some $\tilde{\phi}_m\in V_{b,m}^1$ and estimate
    the resulting complexity depending on the
    approximation error $\delta>0$.
    
    \item    We represent (for $\bar m\leq m$) or approximate (for $\bar m>m$)
    $\varphi_{l,j}$ by $\tilde{\varphi}_{l,j}:=b^{l/p}\tilde{\phi}_{m}(b^l\cdot-j)$
    and estimate the resulting error and complexity.
    
    \item    We represent or approximate the tensor product
    \begin{align*}
        \varphi_\lambda=\varphi_{l_1,j_1}\otimes\ldots\otimes
        \varphi_{l_D,j_D}\quad\text{by}\quad
        \tilde{\varphi}_\lambda:=\tilde{\varphi}_{l_1,j_1}\otimes\ldots\otimes
        \tilde{\varphi}_{l_D,j_D},
    \end{align*}
    and estimate the resulting error and complexity.
    
    \item    Finally, the sum
    $\varphi_n=\sum_{\lambda\in\Lambda}c_\lambda\varphi_\lambda$
    is represented or approximated as
    $\tilde{\varphi}_n=\sum_{\lambda\in\Lambda}c_\lambda\tilde{\varphi}_\lambda$,
    where once again the error and complexity can be estimated with the
    previous steps.
\end{enumerate}

\begin{lemma}[Cardinal B-Splines]\label{lemma:bsplines}
    Let $\phi_{\bar m}:\R\rightarrow\R$ be a cardinal
    B-spline of polynomial degree $\bar m$. Then,
    $\phi_{\bar m}|_{[k,k+1)}$ is a polynomial of degree at most
    $\bar m$ for any $k\in\Z$.
    Consequently, if $\bar m\leq m$, then
    $$\phi_{\bar m,k} := \phi_{\bar m}(\cdot + k)|_{[0,1)}\in V^1_{b,0,m}=\Pm.$$
    
    If $\bar m>m$, fix $L_\delta\in\N$ and divide $[0,1)$ into
    $b^ {L_\delta}$ intervals $I_j:=[b^{-L_\delta}j,b^{-L_\delta}(j+1))$,
    $j=0,\ldots,b^{L_\delta}-1$.
    Let $P_j(f)$ denote the near-best polynomial approximation
    of $f$ in $L^p$ over $I_j$ with polynomial
    degree $m$ (as previously utilized in \cref{eq:decomp}),
    set to zero outside $I_j$.
    For $k\in \{0,\hdots,\bar m\}$, let $\tilde{\phi}_{m,k}:=\sum_{j=0}^{b^{L_\delta}-1}P_j(\phi_{\bar m,k})$.
    Then, we have
    \begin{enumerate}[label=(\roman*)]
        \item\label{ttranksbspline} it holds $\tilde{\phi}_{m,k}\in V^1_{b,L_\delta,m}$ with
        TT-ranks bounded by $\bar m+1$;
        
        \item\label{errorbspline} the approximation error is bounded as
        \begin{align*}
            \norm{\phi_{\bar m,k}-\tilde{\phi}_{m,k}}[L^p(I)]\leq c
            b^{-L_\delta(m+1)}\snorm{\phi_{\bar m,k}}[B^{m+1}_p(L^p(I))],
            \quad I:=[0,1),
        \end{align*}        
        with a constant $c>0$ depending only on
        $m$. 
    
        \item\label{complbspline} the complexities of  $\tilde{\phi}_{m,k}$ 
        are bounded as
         such that
        \begin{align*}
            \sparse(\tilde{\phi}_{m,k})&\leq\full(\tilde{\phi}_{m,k})\leq b^2+b(\bar m+1)^2(L_\delta-1)
            +(m+1)^2
            .
        \end{align*}
    \end{enumerate}

    This implies that to ensure an approximation accuracy
    $\delta>0$, we can set
    \begin{align}\label{eq:Ldelta}
        L_\delta:=\left\lceil
        \frac{1}{m+1}\left|
        \log_b\left[\frac{c\snorm{\phi_{\bar m,k}}[B^{m+1}_p(L^p(I))]}{\delta}
        \right]\right|\right\rceil,
    \end{align}
    i.e., the encoding complexity of $\tilde{\phi}_{m,k} := \tilde \phi_{m}(\cdot+k)$
    depends logarithmically on $\delta^{-1}$.
\end{lemma}
\begin{proof}
    Parts \ref{ttranksbspline} and \ref{complbspline} follow
    from \cite[Section 4.2]{ali2023approximation}.
    Part \ref{errorbspline} follows from, e.g.,
    \cite[Section 4]{devore1988interpolation}.
\end{proof}

Next, we want to deduce a representation or approximation
for a dilated spline $\varphi_{l,j}$.
The operations of dilation and translation
are, in a sense, inverse to tensorization: while the former ``compresses''
and shifts the function, the latter ``zooms in'' onto a piece. 
Let $\phi_{\bar m}:\R\rightarrow\R$ be a
cardinal B-spline as above.
Then, $\phi_{\bar m}$ is supported on
$[0,\bar m+1]$ where it is polynomial
over each integer interval $(j,j+1)$, $j=0,\ldots,\bar m$.
Define $ {\varphi}_{l,j,k}:[0,1)\rightarrow\R$
for $j=0,\ldots,b^l-1$
and $k=0,\ldots,\bar m$ as
    \begin{align*}
         {\varphi}_{l,j,k}(x):=
        \begin{cases}
            b^{l/p}\phi_{\bar m,k}(b^lx-j),&\quad\text{for}\quad x\in[b^{-l}j,b^{-l}(j+1)),\\
            0,&\quad\text{otherwise}.
        \end{cases}
    \end{align*}
    Then,
    \begin{align*}
        \varphi_{l,j}(x)=b^{l/p}\phi_{\bar m}(b^lx-j)
        =\sum_{k=0}^{\bar m} {\varphi}_{l,j+k,k}(x).
    \end{align*}
    
    The utility of $ {\varphi}_{l,j, k}$ is that we have the following
    identity: for any $l\in\N_{\geq 0}$,
    $  {\bs\varphi}_{l,j,k}:= T^1_{b,l}( {\varphi}_{l,j,k})
    \in\bs V^1_{b,l,\bar m}$ with
    \begin{align}\label{eq:rankone}
         {\bs\varphi}_{l,j,k}=
        (b^{1/p}\delta_{j_1})\otimes\cdots\otimes
        (b^{1/p}\delta_{j_l})\otimes\phi_{\bar m,k},
    \end{align}
    and $j=\sum_{k=1}^lj_kb^{l-k}$, \revZ{where $\delta_j$ is the $j$-th canonical vector in $\R^b$.} 
    Therefore, 
    we have a simple representation for the tensorization
    of $ {\varphi}_{l,j,k}$. Hence, to obtain a tensorization
    of $(\varphi_{l,j})_{|[0,1)}$ we can 
    sum at most
    $\bar m+1$ tensorizations of $ {\varphi}_{l,j,k}$. See also
    \cite[Corollary 2.5 and Lemma 2.6]{ali2023approximation} for more details.

\begin{lemma}[Dilated Splines]\label{lemma:dilated}\leavevmode
    \begin{enumerate}[label=(\roman*)]
        \item     For $\bar m\leq m$, we have $\varphi_{l,j}\in V^1_{b,L,m}$
                    for any $L\geq l$ with
                    TT-ranks bounded as $r_\nu\leq\bar m+1$
                    for $1\leq \nu 
                    \leq L$.
                    Moreover,
                    \begin{align*}
                        \sparse(\varphi_{l,j}
                        )&\leq\full(\varphi_{l,j})
                        \leq b^2+b(\bar m+1)^2(l-1)+
                        b(\bar m+1)^2(L-l)+(m+1)^2
                        .
                    \end{align*}
        \item   For $\bar m>m$, we use the
        approximations 
        $\tilde{\phi}_{m,k}$ from \Cref{lemma:bsplines} and set
        $$\tilde{\varphi}_{l,j}:=\sum_{k=0}^{\bar m}
        b^{l/p}\tilde{\phi}_{\bar m,k}(b^l\cdot-j)\indicator_{[b^{-l}j,
        b^{-l}(j+1))}.$$
        Then, due to the $L^p$-normalization, for $I:=[0,1)$
        \begin{align*}
            \norm{\varphi_{l,j}-\tilde{\varphi}_{l,j}}[L^p(I)]\leq\delta,
        \end{align*}
        and $\tilde{\varphi}_{l,j}\in V^1_{b,L,m}$
        for any $L\geq l+L_\delta$, \revZ{with $L_\delta$ defined in \eqref{eq:Ldelta}}.
        The TT-ranks are bounded as before
        $r_\nu\leq\bar m+1$
        for $1\leq \nu\leq l+L_\delta$ and $r_\nu\leq m+1$
        for $l+L_\delta<\nu\leq L$, and
                    \begin{align*}
                        \sparse(\tilde{\varphi}_{l,j})&\leq\full(\tilde{\varphi}_{l,j}) \\
                        &
                        \leq b^2+b(\bar m+1)^2(l+L_\delta-1)
                        +b(m+1)^2(L-l-L_\delta)+(m+1)^2
                        .
                    \end{align*}
    \end{enumerate}
\end{lemma}
\begin{proof}
    Follows from \cite[Corollary 2.5 and Lemma 2.6]{ali2023approximation}.
\end{proof}

In a third step, we transition to the $D$-dimensional case
by using the above and estimating the error and complexity
of tensor products of dilated splines.
\begin{lemma}[Tensor Products of Dilated Splines]\label{lemma:tensorproducts}
    The following complexity bounds hold for $m>0$.

    Let $\varphi_\lambda=
    \varphi_{l_1,j_1}\otimes\ldots\otimes\varphi_{l_D,j_D}$
    be a tensor product of dilated cardinal
    B-splines of polynomial degree at most $\bar m$.
    Set $l:=\max_{\nu}l_\nu$.
    \begin{enumerate}[label=(\roman*)]
        \item If $\bar m\leq m$, then clearly
        $\varphi_\lambda|_{[0,1)^D}\in V_{b,L,m}^D$
        for any $L\geq l$.
        Moreover, the TT-ranks are bounded as 
        $r_\nu\leq(\bar m+1)^D$
        for $1\leq \nu 
        \leq LD$.
        The encoding complexity
        can thus be estimated as
                    \begin{align*}
                        \sparse(\varphi_{\lambda})&\leq\full(\varphi_{\lambda})
                       \\
                       & \leq b^2+b(\bar m+1)^{2D}(lD-1)+
                        b(\bar m+1)^{2D}(L-l)D+
                        (4/3)(\bar m+1)^{2D}
                        .
                    \end{align*}        
        
        \item If $\bar m>m$, then we approximate $\varphi_\lambda$
        by tensor products of functions from \Cref{lemma:dilated}
        via $$\tilde{\varphi}_\lambda=
    \tilde{\varphi}_{l_1,j_1}\otimes\ldots\otimes\tilde{\varphi}_{l_D,j_D}.$$
    The resulting error can be bounded as
    \begin{align}\label{eq:tensorerror}
        \norm{\varphi_\lambda-\tilde{\varphi}_\lambda}[L^p]\leq
        D(\delta+\norm{\phi_{\bar m}}[L^p])^{D-1}\delta.
    \end{align}
    For any $L\geq l+L_\delta$, \revZ{with $L_\delta$ defined in \eqref{eq:Ldelta}},
    $\tilde{\varphi}_\lambda\in V^D_{b,L,m}$
    and the TT-ranks are bounded by $(\bar m+1)^D$
    for $1\leq\nu\leq (l+L_\delta)D$,
    by $(m+1)^D$ for $(l+L_\delta)D<\nu\leq LD$.
    Consequently the encoding complexity
    can be estimated as
                    \begin{align*}
                        \sparse(\tilde{\varphi}_\lambda)&\leq\full(\tilde{\varphi}_\lambda) \\
                  &      \leq b^2+b(\bar m+1)^{2D}(l+L_\delta-1)D
                        +b(m+1)^{2D}(L-l-L_\delta)D
                        +(4/3)(m+1)^{2D}
                        .
                    \end{align*}    
    \end{enumerate}
\end{lemma}
\begin{proof}
    For \cref{eq:tensorerror}, note that we can expand the error as
    \begin{align*}
        \varphi_\lambda-\tilde{\varphi}_\lambda&=
        (\varphi_{l_1,j_1}-\tilde{\varphi}_{l_1,j_1})
        \otimes\varphi_{l_2,j_2}\otimes\ldots\otimes\varphi_{l_D,j_D}\\
        &+\varphi_{l_1,j_1}\otimes(\varphi_{l_2,j_2}-\tilde{\varphi}_{l_2,j_2})
        \otimes\varphi_{l_3,j_3}\otimes\ldots\otimes\varphi_{l_D,j_D}+\ldots,
    \end{align*}
    and applying a triangle inequality and the fact
    that each $\varphi_{l_\nu,j_\nu}$ is normalized
    in $L^p$ yields \cref{eq:tensorerror}.
    
    Next, consider the case $\bar m\leq m$.
    Let $\mu=1,\ldots,L-1$ and $\nu=1,\ldots,D$. Then,
    from \Cref{lemma:dilated}, we know that
    any $\varphi_{l_\nu,j_\nu}$ admits a tensorization such that
    for $x=t_{b,L}^1(i_\nu^1,\ldots,i_\nu^L,\bar x_\nu)$
    \begin{align*}
        \varphi_{l_\nu,j_\nu}(x_\nu)=
        \bs\varphi_{l_\nu,j_\nu}(i_\nu^1,\ldots,i_\nu^L,\bar x_\nu)
        =\sum_{k=1}^{r_\mu}\bs v^\nu_k(i_\nu^1,\ldots,i_\nu^\mu)
        \bs w^\nu_k(i_\nu^{\mu+1},\ldots,i_\nu^L),
    \end{align*}
    for some $\bs v_k$ and $\bs w_k$, where $r_\mu$ is at most $\bar m+1$.
    Similarly for $\mu=L$.

    Thus, taking tensor products and for some
    $\mu=1,\ldots,L-1$, $\nu=1,\ldots,D-1$,
    we can write
    \begin{align*}
        &\varphi_{\lambda}(x)=
        \bs\varphi_{\lambda}(i_1^1,\ldots,\bar x_1,\ldots,\bar x_D)
        =\\
        &\sum_{k_1,\ldots,k_D=1}^{r_1,\ldots,r_D}
        \left(\prod_{\eta=1}^{\nu}\bs v^\eta_{k_\eta}(i_\eta^1,\ldots,i_{\eta}^\mu)\right)
        \left(\prod_{\eta=\nu+1}^{D}\bs v^\eta_{k_\eta}(i_\eta^1,\ldots,i_\eta^{\mu-1})\right)\\
        &\left(\prod_{\eta=1}^{\nu}\bs w^\eta_{k_\eta}(i_{\eta}^{\mu+1},\ldots,i_\eta^L)\right)
        \left(\prod_{\eta=\nu+1}^{D}\bs w^\eta_{k_\eta}(i_\eta^{\mu+1},\ldots,i_\eta^L)\right),
    \end{align*}
    and similarly for the cases $\nu=D$, $\mu=L$.
    The number of summands is bounded by at most $(m+1)^D$
    and thus the rank bound follows. The case $\bar m>m$ can be treated similarly
    by using \Cref{lemma:dilated}, replacing $m$ with $\bar m$
    and $L$ with $L+L_\delta$.
\end{proof}

Finally, in a fourth step, we want to bound the complexity
of representing or approximating
linear combinations of tensor products of dilated splines.
We distinguish two cases:
\begin{enumerate}
    \item the linear approximation case, where we consider a sum of
    all dilated
    splines on some fixed level,
    
    \item and the nonlinear approximation case, where we consider
    an arbitrary sum of $n$ terms.
\end{enumerate}
The latter case is a relatively straightforward application of
\Cref{lemma:tensorproducts},
where we initially assume the maximal level is given.
Later in \Cref{sec:dirinv} we will derive bounds for the maximal level
depending on the approximation accuracy.

Of course, the same bounds could be applied to the linear case.
However, as we will show next, the complexity estimates can be
improved exploiting the fact that the sum contains
\emph{all} splines on a certain level.

\begin{theoremEnd}{theorem}[Linear Approximation with Splines]\label{thm:linsplinestn}
    The following complexity bounds hold for $m>0$.
    \begin{enumerate}[label=(\roman*)]
%
        
        \item If $\bar m\leq m$ and for a given
        smoothness multi-index
        $\bs\alpha=(\alpha_1,\ldots,\alpha_D)$,
        let $\varphi_n:=
        \sum_{|\lambda_\A|\leq L}d_{\lambda_\A,p}\varphi_{\lambda_\A}$
        where the number of terms satisfies
        $n\sim b^{LD\underline{\alpha}/s_\A}$ \revQ{(with constants depending on $m$ and $D$ but not on $n$ and $L$)}. Then,
        $\varphi_n\in V^D_{b,L,m}$ with complexity bounds
        \begin{align*}
            \sparse(\varphi_n)&\leq\full(\varphi_n)\lesssim
            (1+(\bar m+1)^{2D})n^{c(\bs\alpha,D)}
            +b^2+(4/3)(m+1)^{2D},
        \end{align*}
        where the exponent $c(\bs\alpha,D)$ satisfies $1\leq c(\bs\alpha,D)<2D/(D+1)<2$
        for $D\geq 2$,
        see proof for precise form. \rev{It depends on the degree of
        anisotropy. For the isotropic case
        $\bs\alpha=(s_\I,\ldots,s_\I)$, $s_\A = \underline{\alpha}=  s_\I $ and $c(\bs\alpha,D)=1$. }
        
        If $\bar m>m$, we consider instead
        $\tilde{\varphi}_n:=
        \sum_{|\lambda_\A|\leq L}d_{\lambda_\A,p}\tilde{\varphi}_{\lambda_\A}$.
        Then, $\tilde{\varphi}_n\in V^D_{b,L+L_\delta,m}$, \revZ{where $L_\delta$ is defined in \eqref{eq:Ldelta}}, with complexity bounds
        \begin{align*}
            \sparse(\tilde{\varphi}_n)&\leq\full(\tilde{\varphi}_n) \\&\lesssim
            (1+(\bar m+1)^{2D})n^{c(\bs\alpha,D)}+L_\delta D(\bar m+1)^{2D}
            +b^2+(4/3)(m+1)^{2D}.
        \end{align*}

        \item If $\bar m\leq m$, let $\varphi_n:=
        \sum_{|\lambda_\M|_1\leq L}d_{\lambda_\M,p}\varphi_{\lambda_\A}$
        where the number of terms satisfies
        $n\sim L^{D-1}b^L$ \revQ{(with constants depending on $m$ and $D$ but not on $n$ and $L$)}. Then,
        $\varphi_n\in V^D_{b,L,m}$ with complexity bounds
        \begin{align*}
            \sparse(\varphi_n)&\leq\full(\varphi_n)\lesssim
            (1+(\bar m+1)^{2D})n^{c(D)}
            +b^2+(4/3)(m+1)^{2D},
        \end{align*}
        where the factor satisfies
        $1< c(D)<2D/(D+1)<2$ for $D\geq 2$, see proof for precise form.             
        
        If $\bar m>m$, we consider instead
        $\tilde{\varphi}_n:=
        \sum_{|\lambda_\M|_1\leq L}d_{\lambda_\M,p}\tilde{\varphi}_{\lambda_\M}$.
        Then, $\tilde{\varphi}_n\in V^D_{b,L+L_\delta,m}$, \revZ{where $L_\delta$ is defined in \eqref{eq:Ldelta}}, and with complexity bounds
        \begin{align*}
            \sparse(\tilde{\varphi}_n)&\leq\full(\tilde{\varphi}_n)\\&\lesssim
            (1+(\bar m+1)^{2D})n^{c(D)}+L_\delta D(\bar m+1)^{2D}
            +b^2+(4/3)(m+1)^{2D}
            .
        \end{align*}      
    \end{enumerate}
\end{theoremEnd}
\begin{proofEnd}\leavevmode
    \begin{enumerate}[label=(\roman*)]
        \item Let $\bar m\leq m$ and consider $\nu=1,\ldots,L$. Then,
        \begin{align*}
            \dim\linspan\left\{
            \bs\varphi_n(i_1^1,\ldots,i_D^\nu,\cdot):\;
            (i_1^1,\ldots,i_D^\nu)\in\{0,\ldots, b-1\}^{\nu D}
            \right\}            
            \leq(\bar m+1)^D
            b^{(L-\nu)D}.            
        \end{align*}
        Hence, overall
        \begin{align*}
            r_{\nu D}\leq \min\{b^{\nu D},(\bar m+1)^Db^{(L-\nu)D}\}.
        \end{align*}
        For the intermediate ranks $r_\mu$, $\nu D<\mu<(\nu+1)D$,
        we can apply \Cref{lemma:admissranks}, i.e.,
        \begin{align*}
            r_\mu\leq\{b^\mu,(\bar m+1)^Db^{(L-\nu-1)D}b^{(\nu+1)D-\mu}\}
            =\min\{b^\mu, (\bar m+1)^Db^{LD-\mu}\}.
        \end{align*}
        Thus, overall we can balance the two terms to obtain
        the complexity bound for
        \begin{align*}
            \full(\varphi_n)&\leq b^2+
            b\sum_{\mu=2}^{\lceil LD/2\rceil}b^{2\mu-1}+
            (\bar m+1)^{2D}\sum_{\mu=\lceil LD/2\rceil+1}^{LD}b^{2LD-2\mu+1}+
            (m+1)^{2D+1}\\
            &\lesssim (1+(\bar m+1)^{2D})n+b^2+(m+1)^{2D+1}
            .
        \end{align*}
        
        For $\bar m>m$, the same arguments apply
        for $\mu=1,\ldots,LD$ and, for $LD<\mu\leq (L+L_\delta)D$,
        we have $r_\mu\leq (\bar m+1)^D$.
        Thus, overall
        \begin{align*}
            \full(\tilde{\varphi}_n)&\lesssim (1+(\bar m+1)^{2D})n
            +L_\delta D(\bar m+1)^{2D}+b^2+(m+1)^{2D+1}
           .
        \end{align*}

        \item Most of the arguments carry over, so we only consider
        $\bar m\leq m$ and the necessary adjustments.
        The key difference to the isotropic case is that
        for the anisotropic multilevel $l_\nu=l_\nu(L)$ we obtain
        \begin{align*}
            \dim\linspan\left\{
            \bs\varphi_n(i_1^1,\ldots,i_D^\nu,\cdot):\;
            (i_1^1,\ldots,i_D^\nu)\in\{0,1\}^{\nu D}
            \right\}            
            \leq(\bar m+1)^D
            b^{\sum_{k=1}^D(l_k-\nu)_+},            
        \end{align*}
        with $x_+:=\max\{0,x\}$.
        Hence, deriving optimal bounds now relies on a more careful
        balancing of the terms in the rank bound
        \begin{align*}
            r_{\nu D}\leq\min\{b^{\nu D},(\bar m+1)^D
              b^{\sum_\nu(l_\nu-\nu)_+}\}.
        \end{align*}
        This ``optimal'' balancing depends on $\bs\alpha$.
        We first derive a general expression and then provide
        a crude but simple bound.
        
        Assume w.l.o.g.\ that the smoothness multi-index is ordered
        as $\underline{\alpha}=\alpha_1 \leq \alpha_2 \leq \ldots \leq \alpha_D$ --
        apply an index permutation $\alpha_{\sigma(\nu)}$ otherwise.
        Consequently,
        $L=l_1 \geq   l_2 \geq \ldots \geq  l_D$.
        Then, for any $\nu=1,\ldots,L$, define
        \begin{align*}
            k(\nu):=\min\{\mu=1,\ldots,D:\;
            \alpha_\mu\geq (L/\nu)\alpha_1\}.
        \end{align*}
        Finally, set
        \begin{align*}
            \nu^*:=\nu^*(L,\bs\alpha,D):=
            \min\left\{
                \nu=1,\ldots,L:\;
                \nu>
                \frac{L\underline{\alpha}\sum_{\mu=1}^{k(\nu)}
                \alpha_\mu^{-1}}{D+k(\nu)}
            \right\}-1.
        \end{align*}
        Then, the encoding complexity can be bounded as
        \begin{align*}
            \full(\varphi_n)&\lesssim
            (1+(\bar m+1)^{2D})n^{c(\bs\alpha,D)}
            +b^2+(m+1)^{2D+1},
        \end{align*}
        where
        \begin{align*}
            1\leq c(\bs\alpha,D):=2\nu^*s_\A/(L\underline{\alpha})
            \leq \frac{2D}{D+1}<2.
        \end{align*}

        \item As in the anisotropic case, the key difference to the isotropic case
        is the dimension of the minimal subspace.
        For the multilevel $l=(l_1,\ldots,l_D)$
        \begin{align*}
            \dim\linspan\left\{
            \bs\varphi_n(i_1^1,\ldots,i_D^\nu,\cdot):\;
            (i_1^1,\ldots,i_D^\nu)\in\{0,1\}^{\nu D}
            \right\}            
            \leq(\bar m+1)^D
            \sum_{|l|_1=L}b^{\sum_{k=1}^D (l_k-\nu)_+},            
        \end{align*}
        and hence for the rank bound we balance the expression
        \begin{align*}
            r_{\nu D}\leq\min\{b^{\nu D},(\bar m+1)^D
              \sum_{|l|_1=L}b^{\sum_k(l_k-\nu)_+}\}.
        \end{align*}
        
        We can write
        \begin{align*}
             \sum_{|l|_1=L}b^{\sum_k(l_k-\nu)_+}=
             C_\#\sum_{\mu=1}^{L-\nu}(b^\mu+\nu),
        \end{align*}
        where $C_\#=C_\#(L, D)$ is the number of distinct ways in which
        all partitions of $L$ into $D$ integers (including $0$)
        can be written. I.e., more precisely, we have in total
        \begin{align*}
            \sum_{|l|_1=L}=C_{\mathrm{total}}\sim L^{D-1}.
        \end{align*}
        This sum can be decomposed into $P(L,D)$
        -- the partition number of $L$ into $D$ integers
        -- times the number of distinct (ordered) ways to represent
        \emph{all} partitions, i.e.,
        \begin{align*}
            C_{\mathrm{total}}=P(L,D)C_{\#}(L,D).
        \end{align*}
        Then, as before we can set
        \begin{align*}
            \nu^*:=\nu^*(L,D):=
            \min\left\{
                \nu=1,\ldots,L:\;
                \nu>\frac{D\log_b(\bar m+1)+\log_b(C_\#)+L+1}{D+1}
            \right\}-1,
        \end{align*}
        and estimate the complexity as
        \begin{align*}
            \full(\varphi_n)&\lesssim
            (1+(\bar m+1)^{2D})n^{c(D)}
            +b^2+(m+1)^{2D+1}
            ,
        \end{align*}
        where
        \begin{align*}
            1<c(D):=2\nu^*D/L
            \leq \frac{2D}{D+1}<2.
        \end{align*}
    \end{enumerate}
\end{proofEnd}

\begin{theorem}[Nonlinear Approximation with Splines]\label{thm:nonlinsplinestns}
    For $\bar m\leq m$,
    let $\varphi_n=\sum_{\lambda\in\Lambda}d_{\lambda,p}\varphi_{\lambda}$
    be an arbitrary\footnote{Isotropic, anisotropic or mixed.}
    $n$-term expansion with $\#\Lambda\leq n$
    and maximal level $L$.
    Then, $\varphi_n\in V^D_{b,L,m}$ with complexity bounds
    \begin{align*}
        \full(\varphi_n)&\leq b^2+b(m+1)^{2D}n^2LD+(4/3)(m+1)^{2D},\\
        \sparse(\varphi_n)&\leq[b^2+b(m+1)^{2D}LD+(4/3)(m+1)^{2D}]n
        .
    \end{align*}
    
    If $\bar m>m$, we consider
    $\tilde{\varphi}_n=\sum_{\lambda\in\Lambda}d_{\lambda,p}
    \tilde{\varphi}_{\lambda}$ with $\#\Lambda\leq n$ and
    maximal level $L$.
    Then,
     $\tilde{\varphi}_n\in V^D_{b,L+L_\delta,m}$ with complexity bounds
    \begin{align*}
        \full(\tilde{\varphi}_n)&\leq b^2+b(\bar m+1)^{2D}n^2LD+(4/3)(m+1)^{2D},\\
        \sparse(\tilde{\varphi}_n)&\leq[b^2+b(\bar m+1)^{2D}LD+(4/3)(m+1)^{2D}]n
        .
    \end{align*}
\end{theorem}
\begin{proof}
    The proof is an application of
    \Cref{lemma:tensorproducts,lemma:sparsity}.
\end{proof}


\subsection{Direct Embeddings
for Tensor Networks}\label{sec:dirinv}
The results of the previous section can be directly applied
to infer approximation rates with TTs for
Besov spaces introduced in \Cref{sec:review}. All constants in the inequalities for $n$ 
\revQ{ 
can be inferred from the previous section, and do not depend on $D$, which allows us to discuss the curse of dimensionality. Since we are concerned only with the asymptotic behavior and to avoid unnecessary technicalities, we do not provide the explicit expressions of these constants.} 

\begin{theorem}[Approximation of $B^{s_\I}_p(L^p(\Omega))$, $AB^{\bs\alpha}_p(L^p(\Omega))$ and $ MB^{s_\M}_p(L^p(\Omega))$ with Tensor Trains]\label{thm:mainlinear}
    Let $\Omega:=[0,1)^D$ and consider the approximation tools
    from \Cref{def:tools} with $S=\Pm$ with arbitrary polynomial degree
    $m\in\N_{\geq 0}$.
    \begin{enumerate}[label=(\roman*)]
%
        
        \item \emph{(Anisotropic Smoothness)}. Let $0<p\leq\infty$, $\bs\alpha\in(\R_{>0})^D$ and $s_\A := s_\A(\bs\alpha)$ the aggregated smoothness.
        For any $f\in AB^{\bs\alpha}_p(L^p(\Omega))$, it holds
        {\begin{align*}
            E(f, \Phi^{\mc F}_n)_{L^p}&\leq Cn^{-s/(c(\bs\alpha,D)D)}
            \snorm{f}[AB^{\bs\alpha}_p(L^p(\Omega))]
         ,
        \end{align*}
        }
        where either $s=s_\A$ if $\bar\alpha\leq\min\{m+1,m+1/p\}$,
        or $0<s<s_\A$ arbitrary if $\bar\alpha>\min\{m+1,m+1/p\}$,
        for $C\sim (\lfloor \overline{\alpha}\rfloor+1)^{2D}$ and
        any $n\gtrsim (m+1)^{2D}$.
        The factor $c(\bs\alpha,D)$ was introduced in \Cref{thm:linsplinestn},
        note that $1/2<1/c(\bs\alpha,D)\leq 1$ for $D\geq 2$.
        
        For the approximation spaces this implies the following
        continuous embeddings: for any $0<q\leq\infty$
        \rev{\begin{align*}
            AB^{\bs\alpha}_q(L^p(\Omega))&\hookrightarrow
            \mc F^{s/(c(\bs\alpha,D)D)}_q(L^p).
        \end{align*}
        }
        \rev{In the special case $\boldsymbol{\alpha} = (s_\I, \dots, s_\I)$, it holds $s_\A = \underline{\alpha} = \bar \alpha= s_\I$, the space $AB^{\bs\alpha}_{q}(L^p(\Omega))$ coincides with the isotropic Besov space $B^{s_\I}_{q}(L^p(\Omega))$, and the factor $c(\bs\alpha,D) = 1$. Therefore, 
         \rev{ \begin{align*}
            B^{s_\I}_q(L^p(\Omega))&\hookrightarrow
            \mc F^{s/D}_q(L^p).
        \end{align*}} 
        }

        \item \emph{(Mixed Dominating Smoothness)}. Let $0<p\leq\infty$ and $s_\M>0$.
        For any $f\in MB^{s_\M}_p(L^p(\Omega))$, it holds
        {\begin{align*}
            E(f, \Phi^{\mc F}_n)_{L^p}&\leq Cn^{-s_\M/c(D)}[(1/c(D))\log_b(n)]^{s_\M(D-1)}
            \snorm{f}[MB^{s}_p(L^p(\Omega))].
        \end{align*}
        }
        where either $s=s_\M$ if $s_\M\leq\min\{m+1,m+1/p\}$,
        or $0<s<s_\M$ arbitrary if $s_\M>\min\{m+1,m+1/p\}$,
        for $C\sim (\lfloor s_\M\rfloor+1)^{2D}$ and
        any $n\gtrsim (m+1)^{2D}$.
        The factor $c(D)$ was introduced in \Cref{thm:linsplinestn},
        note that $1/2<1/c(D)<1$ for $D\geq 2$.
        
        For the approximation spaces this implies the following
        continuous embeddings: for any $0<q\leq\infty$
        and any $0<s<s_\M$
        \rev{\begin{align*}
            MB^{s}_q(L^p(\Omega))&\hookrightarrow
            \mc F^{s/c(D)}_q(L^p).
        \end{align*}}
    \end{enumerate}
\end{theorem}
\begin{proof}
    For any given $f$, we take auxiliary
    $\varphi_\lambda$ of \revZ{polynomial
    degree $\bar m$ satisfying \eqref{eq:mbar_high_enough}, which ensures that spline approximation is rate optimal for the considered space}.  
    If $\bar m\leq m$, we can represent $\varphi_\lambda$ exactly
    as a TT and estimate the resulting complexity.
    Otherwise, if $\bar m>m$,
    we approximate with $\tilde{\varphi}_\lambda$,
    apply a triangle and Hölder inequalities
    \begin{align*}
        \norm{f-\sum d_{\lambda,p}(f)\tilde{\varphi}_\lambda}[p]\leq &
        \norm{f-\sum d_{\lambda,p}(f)\varphi_\lambda}[p] +\\
        &D(\delta+\norm{\phi_{\bar m}}[L^p])^{D-1}\delta
        (\sum b^{s_\I p|\lambda|}|d_{\lambda,p}|^p)^{1/p}n^{1/q(p)}.
    \end{align*}        
    The results follow from
    \Cref{thm:chars},    
    \Cref{thm:classiclinear},
    \cref{eq:tensorerror} and \Cref{thm:linsplinestn}.
\end{proof}

Next, we turn to nonlinear approximation and
the spaces $B^{s_\I}_\tau(L^\tau(\Omega))$ above
the diagonal in \Cref{fig:DeVore}.
To this end, we need to estimate the maximal
level $L$, analogously to \cite{ali2023approximation}.

\begin{lemma}[Maximal Level]\label{lemma:maxlevel}
    Let $\Omega:=[0,1)^D$
    and $\varphi_\lambda$ be order-$D$
    tensor products of dilated one-dimensional
    splines $\varphi_{\bar m}$
    of polynomial degree at most $\bar m$.
    \begin{enumerate}[label=(\roman*)]
%
        \item \emph{(Anisotropic Smoothness)}. Let $0<p<\infty$ and $f\in AB^{\bs\alpha}_\tau(L^\tau(\Omega))$ with $\bs\alpha\in(\R_{>0})^D$.
        Let $0<\bar\alpha<\min\{\bar m+1,\bar m+1/p\}$ 
        with $0<\tau<p$ such that
        \begin{align}\label{eq:excess}
            s_\A/D>1/\tau-1/p,
        \end{align}
        where $s_\A := s_\A(\bs \alpha)$ is the aggregated smoothness.
        Assume $\varphi_n=\sum_{\lambda_\A\in\Lambda}d_{\lambda_A,p}(f)
        \varphi_{\lambda_\A}$, $\#\Lambda\leq n$ is an
        $n$-term approximation to $f$ such that
        \begin{align}\label{eq:ntermeps}
            \norm{f-\varphi_n}[p]\leq \varepsilon,
        \end{align}
        for an arbitrary $\varepsilon>0$.
        Then, w.l.o.g. we can assume
        \begin{align}\label{eq:leveliso}
            \max_{\lambda_\A\in\Lambda}|\lambda_\A|\leq\rho_\A(\varepsilon,n):=
            \left|\frac{\underline{\alpha}C_{\Phi_\A}\norm{\phi_{\bar m}}[p]^D}{s_\A(s_\A-D[1/\tau-1/p])}
            \log_b\left(
            \frac{\varepsilon}{2\norm{f}[AB^{\bs\alpha}_\tau(L^\tau(\Omega))]
            n^{1/q}}\right)\right|,
        \end{align}
        where $q$ is the Hölder conjugate of $\tau$
	    \begin{align*}
	        q:=
	        \begin{cases}
	            (1-1/\tau)^{-1},&\quad\text{if }\tau>1,\\
	            \infty,&\quad\text{otherwise},
	        \end{cases}
	    \end{align*}
	    and the constant $C_{\Phi_\A}$
	    depends on the equivalence constants from
	    \Cref{thm:chars}.
\\
\rev{In the special case $\boldsymbol{\alpha} = (s_\I, \dots, s_\I)$, it holds $s_\A = \underline{\alpha} = \bar \alpha= s_\I$ and $AB^{\bs\alpha}_{q}(L^\tau(\Omega))$ coincides with the isotropic Besov space $B^{s_\I}_{q}(L^\tau(\Omega))$.}
        
        \item \emph{(Mixed Dominating Smoothness)}. Let $1<p<\infty$ and $f\in MB^{s_\M}_\tau(L^\tau(\Omega))$.
        Let $0<s_\M<\min\{\bar m+1,\bar m+1/p\}$ 
        with $0<\tau<p$ such that
        \begin{align*}
            s_\M>1/\tau-1/p.
        \end{align*}
        Assume $\varphi_n=\sum_{\lambda_\M\in\Lambda}d_{\lambda_M,p}(f)
        \varphi_{\lambda_\M}$, $\#\Lambda\leq n$ is an
        $n$-term approximation to $f$ such that
        \begin{align*}
            \norm{f-\varphi_n}[p]\leq \varepsilon,
        \end{align*}
        for an arbitrary $\varepsilon>0$.
        Then, w.l.o.g. we can assume
        \begin{align*}
            \max_{\lambda_\M\in\Lambda}|\lambda_\M|_1\leq\rho_\M(\varepsilon,n):=
            \left|\frac{C_{\Phi_\M}\norm{\phi_{\bar m}}[p]^D}{s_\M-(1/\tau-1/p)}\log_b\left(
            \frac{\varepsilon}{2\norm{f}[MB^{s_\M}_\tau(L^\tau(\Omega))]
            n^{1/q}}\right)\right|.
        \end{align*}
    \end{enumerate}
\end{lemma}

\begin{proof}
\rev{For the proof, we only consider the isotropic case $B^{s_I}_\tau(L^\tau) = AB^{\bs{\alpha}}_\tau(L^\tau)$, with $\bs{\alpha} = (s_\I, \dots, s_\I)$ and $s_A = \underline{\alpha} = \bar \alpha = s_\I$. The more general anisotropic case and the mixed case follow analogously. }
    Let $\varphi_n$ be an $n$-term approximation satisfying
    \cref{eq:ntermeps}. Set
    \begin{align*}
        \tilde\Lambda:=\{\lambda_\I\in\Lambda:\;|\lambda_\I|\leq\rho_\A(\varepsilon,n)\},
    \end{align*}
    and define $\bar{\varphi}_n:=\sum_{\lambda_\I\in\tilde{\Lambda}}d_{\lambda_\I,p}(f)
    \varphi_{\lambda_\I}$.
    The relationship between coefficients normalized in different $L^p$-norms
    is as follows
    \begin{align*}
        d_{\lambda_\I,p}=b^{-|\lambda_\I|D(1/p-1/\tau)} d_{\lambda_\I,\tau}.
    \end{align*}
    Then, using \Cref{thm:embeds,thm:chars},
    excess regularity from \cref{eq:excess} and a Hölder inequality,
    we get
    \begin{align*}
        \norm{f-\bar{\varphi}_n}[p]&\leq
        \norm{f-\varphi_n}[p]+\norm{\phi_{\bar m}}[p]^D
        \sum_{\lambda_\I\in\Lambda\setminus
        \tilde{\Lambda}}|d_{\lambda_\I,p}|
        \\
        &=\varepsilon+\norm{\phi_{\bar m}}[p]^D
        \sum_{\lambda_\I\in\Lambda\setminus
        \tilde{\Lambda}}|d_{\lambda_\I,\tau}|
        b^{-|\lambda_\I|D(1/p-1/\tau)}b^{|\lambda_\I|s_\I}
        b^{-|\lambda_\I|s_\I}\\
        &\leq
        \varepsilon+\norm{\phi_{\bar m}}[p]^D
        \left(\sum_{\lambda_\I\in\Lambda\setminus
        \tilde{\Lambda}}|d_{\lambda_\I,\tau}|
        b^{|\lambda_\I|s_\I\tau}\right)^{1/\tau}
        \left(\sum_{\lambda_\I\in\Lambda\setminus
        \tilde{\Lambda}}
        b^{-|\lambda_\I|(s_\I-D[1/\tau-1/p])q}
        \right)^{1/q}\\
        &\leq\varepsilon+
        C_{\Phi_\I}\norm{\phi_{\bar m}}[p]^D
        \norm{f}[B^{s_\I}_\tau(L^\tau(\Omega))]
        (\max_{\lambda_\I\in\Lambda\setminus
        \tilde{\Lambda}} b^{-|\lambda_\I|(s_\I-D[1/\tau-1/p])})n^{1/q}
        \leq 2\varepsilon.
    \end{align*}
    This yields \cref{eq:leveliso}.
\end{proof}

For a quasi-normed space $X$ and a quasi-semi-normed space
$Y$ with $Y\hookrightarrow X$,
we use $(X,Y)_{\theta,q}$ to denote the
\emph{real $K$-interpolation space},
for $0<\theta<1$, $0<q\leq\infty$ and $q=\infty$ if $\theta=1$.
With the above preparations we now conclude with
approximation rates for Besov spaces $B^{s_\I}_q(L^\tau(\Omega))$ (and variants for anisotropic and mixed dominating smoothness)
and corresponding continuous embeddings.

\begin{theorem}[Approximation of $B^{s_\I}_q(L^\tau(\Omega))$, $AB^{s_\A}_q(L^\tau(\Omega))$ and  $MB^{s_\M}_q(L^\tau(\Omega))$  with Tensor Trains]\label{thm:direct}
    Let $\Omega:=[0,1)^D$.
    \begin{enumerate}[label=(\roman*)]
%
        
        \item \emph{(Anisotropic Smoothness)}. For $0<\tau<p<\infty$, $\bs\alpha\in(\R_{>0})^D$ and aggregated smoothness $s_\A := s_\A(\bs\alpha)$ such that
         $s_\A/D>1/\tau-1/p>0$ and any $0<q\leq\tau$
        it holds
        {
        \begin{alignat*}{2}
            E(f,\PhiF_n)_{L^p} &\lesssim n^{-{s/{(2D)}}}\snorm{f}[AB^{\bs\alpha}_q(L^\tau(\Omega))],\\
            E(f,\PhiS_n)_{L^p} &\lesssim n^{-{s/{D}}}\snorm{f}[AB^{\bs\alpha}_q(L^\tau(\Omega))],
        \end{alignat*}
        }
        where either $s=s_\A$ if $\bar\alpha\leq\min\{m+1,m+1/p\}$,
        or $0<s<s_\A$ arbitrary if $\bar\alpha>\min\{m+1,m+1/p\}$,
        for $C\sim (\lfloor \overline{\alpha}\rfloor+1)^{2D}$ and
        any $n\gtrsim (m+1)^{2D}$.
        
        For the approximation spaces this implies the following
        continuous embeddings
        \rev{\begin{align*}
            AB^{\bs{\alpha}}_q(L^\tau(\Omega))
            \hookrightarrow\mc S^{s/D}_\infty(L^p)
            \hookrightarrow\mc F^{s/(2D)}_\infty(L^p),
        \end{align*}}
        and
        \rev{\begin{align*}
            (L^p(\Omega),AB^{\bs\alpha}_q(L^\tau(\Omega)))_{\theta/s,\bar q}
            \hookrightarrow\mc S^{\theta/D}_{\bar q}(L^p)
            \hookrightarrow\mc F^{\theta/(2D)}_{\bar q}(L^p),
        \end{align*}
        }
        for any $0<\theta<s$, $0<\bar q\leq\infty$.  \rev{In the special case $\boldsymbol{\alpha} = (s_\I, \dots, s_\I)$, it holds $s_\A = \underline{\alpha} = \bar \alpha= s_\I$ and $AB^{\bs\alpha}_{q}(L^\tau(\Omega))$ coincides with the isotropic Besov space $B^{s_\I}_{q}(L^\tau(\Omega))$.}

        \item \emph{(Mixed Dominating Smoothness)}. For $0<\tau<p<\infty$, $s_\M>0$
        with $s_\M>1/\tau-1/p>0$, $p>1$ and
        any $0<q\leq\tau$,
        it holds
        {
        \begin{alignat*}{2}
            E(f,\PhiF_n)_{L^p} &\leq
            Cn^{-s/2}[(1/2)\log_b(n)]^{s_\M(D-1)}
            \snorm{f}[MB^{s_\M}_q(L^\tau(\Omega))],\\
            E(f,\PhiS_n)_{L^p} &\leq Cn^{-s}[\log_b(n)]^{s_\M(D-1)}
            \snorm{f}[MB^{s_\M}_q(L^\tau(\Omega))],
        \end{alignat*}
        }
        where either $s=s_\M$ if $s_\M\leq\min\{m+1,m+1/p\}$,
        or $0<s<s_\M$ arbitrary if $s_\M>\min\{m+1,m+1/p\}$;
        for $C\sim (\lfloor s_\M\rfloor+1)^{2D}$ and
        any $n\gtrsim (m+1)^{2D}$.
        
        For the approximation spaces this implies the following
        continuous embeddings
       \rev{ \begin{align*}
            MB^{s_\M}_q(L^\tau(\Omega))
            \hookrightarrow\mc S^s_\infty(L^p)
            \hookrightarrow\mc F^{s/2}_\infty(L^p),
        \end{align*}}
        and
        \rev{\begin{align*}
            (L^p(\Omega),MB^{s_\M}_q(L^\tau(\Omega)))_{\theta/s,\bar q}
            \hookrightarrow\mc S^\theta_{\bar q}(L^p)
            \hookrightarrow\mc F^{\theta/2}_{\bar q}(L^p),
        \end{align*}
        }
        for any $0<\theta<s$, $0<\bar q\leq\infty$.
    \end{enumerate}
\end{theorem}

\begin{proof}
    For any given $f$, we take auxiliary
    $\varphi_\lambda$ of sufficiently high polynomial
    degree $\bar m$.
    If $\bar m\leq m$, we can represent $\varphi_\lambda$ exactly
    as a TT and estimate the resulting complexity.
    Otherwise, if $\bar m>m$,
    we approximate with $\tilde{\varphi}_\lambda$,
    apply a triangle and Hölder inequalities
    \begin{align*}
        \norm{f-\sum d_{\lambda,p}(f)\tilde{\phi}_\lambda}[p]\leq&
        \norm{f-\sum d_{\lambda,p}(f)\varphi_\lambda}[p] + \\
        &
        D(\delta+\norm{\varphi_{\bar m}}[L^p])^{D-1}\delta
        (\sum b^{s_\I \tau|\lambda|}|d_{\lambda,\tau}|^\tau)^{1/\tau}n^{1/q(\tau)}.
    \end{align*}        
    The results follow from
    \Cref{thm:chars},    
    \Cref{thm:classicbestnterm},
    \cref{eq:tensorerror},
    \Cref{thm:nonlinsplinestns} and \Cref{lemma:maxlevel}.
\end{proof}

\rev{\begin{remark}
The restriction $p>1$ for the case of mixed dominating smoothness
            stems from the Sobolev embeddings in \Cref{thm:embeds},
            which is in turn based on the results of
            \cite{Hansen2012,NonCompact}.         
\end{remark}}

\subsection{No Inverse Embedding}
Functions from any of the spaces
\rev{$\mc F^\alpha_q(L^p)$ and $\mc S^\alpha_q(L^p)$}
do not need to have any smoothness.
Deep tensor networks have the ability to approximate highly irregular
functions with low ranks. Consequently, just like in the one-dimensional case,
if we impose no restrictions on the depth,
TT-approximation spaces cannot be embedded in any \revQ{of the smoothness spaces considered in this work.} 

\begin{theorem}[No Inverse Embeddings]\label{thm:inverse}
    For $\Omega:=[0,1)^D$,
    any $0<p,q\leq\infty$, \rev{$s_\I >0$}
    and any $s>0$, it
    holds
  \rev{  \begin{align*}
        \mc F^s_q(L^p)\not\hookrightarrow
        B^{s_\I}_q(L^p(\Omega)).
    \end{align*}}
\end{theorem}

\begin{proof}
    Follows by similar arguments as in \cite[Theorem 7.1]{ali2023approximation} i.e., one can construct
    a counter-example by taking a $D$-dimensional rank-one tensor product
    of ``sawtooth'' functions, \rev{that is a function efficiently represented in $V_{b,m}^D$ but has ``bad'' Besov regularity.}.
\end{proof}
\revQ{Using Theorem \ref{thm:appspaces}, we deduce a similar statement for the approximation spaces $ \mc S^s_q(L^p)$ of sparse TTs. The statement also holds with  $B^{s_\I}_q(L^p(\Omega))$ replaced by smoothness spaces with anisotropic or mixed regularity, because $B^{s_\I}_q(L^p(\Omega))$  is included 
in   $MB^{s_\M}_q(L^p(\Omega))$   or $AB^{\bs\alpha}_q(L^p(\Omega))$, for sufficiently high $s_\M$ or $\bs\alpha$. 
}

%% file: conclusion.tex

\section{Concluding Remarks}\label{sec:conclusion}
\leavevmode


    For all types of smoothness considered -- isotropic, anisotropic and mixed
    -- the introduced approximation tool based on TTs can reproduce optimal or near to optimal rates of convergence, whatever the local polynomial degree $m$, \rev{even for $m=0$, which corresponds to piecewise contant approximation}. \revY{In other words, it is a highly flexible tool that achieves optimal or near to optimal rates across a wide range of smoothness classes, without the need to adapt it to the type or degree of smoothness, nor to structural properties of multivariate functions such as anisotropy.  \rev{   In a statistical learning setting, the obtained results, along with suitable model selection techniques, allow to prove that minimax adaptive rates can be achieved for  all smoothness classes considered in this work \cite{Michel2022Bernoulli}.    }}
        \\\par 
    
     \paragraph{\bfseries The role of sparsity.} We emphasize the importance of sparsity to obtain
    optimal rates, similarly to \cite{ali2023approximation}, cf.\
    \Cref{thm:mainlinear,thm:direct}.
    In the case of \emph{linear} approximation and \emph{isotropic} smoothness,
    there is no benefit in considering sparse TTs.
    In the case of \emph{linear} approximation and \emph{anisotropic}
    or \emph{mixed} smoothness,
    approximating with TTs with dense cores yields rates that are smaller
    than optimal by a dimension-dependent factor that is strictly
    between $1/2$ and $1$, whereas sparse approximation yields optimal rates.    
    In the case of \emph{nonlinear} approximation and \emph{any} type
    of smoothness, approximating with TTs with dense cores yields rates that are
    worse by a factor of $1/2$, where again sparse approximation recovers
    (near to) optimal rates. \revY{Although TT formats with sparse cores provide rate optimality in all cases, they are not yet suitable for practical computation: primarily because of the absence of efficient and controlled compression algorithms that preserve sparsity. However, TT formats with dense cores are amenable to efficient computation and yield approximation rates that are only a factor of 
$1/2$ worse than those achieved by the corresponding sparse formats. Note, however, that a few recent works have   considered sparse formats in practice, e.g., \cite{Gotte2021Sep,Trunschke2025}.}
        \\\par

    \paragraph{\bfseries The curse of dimensionality.} 
    The curse of dimensionality is still present for all smoothness classes considered:
    in the constants and rate of convergence in the isotropic and anisotropic cases\footnote{For the anisotropic case, the curse of dimensionality may not be present in the rate when assuming sufficient anisotropy.},
    and in the constants and log factors in the mixed case.
    This curse is unavoidable for such classical smoothness spaces. 
    One can instead consider, e.g., $\mc F_q^s(L^p)$ as a model class.
    For fixed $s$ and growing $D\rightarrow\infty$, these model classes
    do not exhibit the curse of dimensionality in the approximation rate.
    It does not exclude the curse of dimensionality in the constant
    for the case  $m>0$,
    i.e., one can have a sequence $f_D\in\mc F_q^s(L^p)$ such
    that $\norm{f_D}[\mc F_q^s(L^p)]\sim 2^D$.    
    Our analysis shows that even if $f_1,f_2\in\mc F_q^s(L^p)$
    with ``small'' norms $\norm{f_1}[\mc F_q^s(L^p)]\sim
    \norm{f_2}[\mc F_q^s(L^p)]\sim 1$, we can still have
    $\norm{f_1+f_2}[\mc F_q^s(L^p)]\sim (m+1)^D$.
    This is due to the structure of the considered
    TTs underlying $\mc F_q^s(L^p)$, where
    on the finest scale the considered function space (of features)
    is $(\Pm)^{\otimes D}$. 
    Note however that the approximation tool with $m=0$ (that corresponds to piecewise constant functions) do not exhibit the curse of dimensionality. \rev{The curse could be also avoided by considering for the local approximation space $S^D$ any space including the constants,  with a dimension scaling only polynomially in $D$. 
    }\\\par
      \rev{
    We further emphasize that the absence of inverse embeddings indicates that the approximation tool considered in this work can efficiently approximate functions beyond classical smoothness spaces,  possibly breaking  the curse of dimensionality (with $m=0$). 
    }
        \\\par 
          
    \paragraph{\bfseries The role of base $b$ and local approximation space $S^D$.} 
    \revY{We introduced a whole family of approximation tools -- and associated TT approximation spaces -- parameterized by the choice of base 
$b$ and the local approximation space 
$S^D$. We showed that classical smoothness spaces are continuously embedded in these TT approximation spaces for any 
$b$
 and any local polynomial degree 
$m$. However, the TT approximation spaces do depend on these parameters, and the approximation of functions outside classical smoothness classes may be highly sensitive to their choice. 
For instance, a fractal function constructed through 
$q$-adic partitions can be effectively represented by TTs via tensorization, provided that 
$b$ is chosen appropriately (in relation to 
$q$). In practical applications, we can observe a clear influence of the parameters 
$b$ and 
$m$
 on the quality of the resulting approximation, suggesting the need for adaptive parameter selection -- for example, through model selection approaches such as those proposed in \cite{Michel2022Bernoulli} within a statistical learning framework.
A thorough analysis of the role of 
$b$
 and the local approximation space 
$S^D$ 
 (potentially beyond polynomial spaces), both from theoretical and practical perspectives, is left for future work.
 }
    \\
    \par 
    
        \paragraph{\bfseries The role of tensorization and tensor formats.} 
    \rev{For the classical smoothness classes considered in this work, near-optimal rates can also be achieved by TNs without relying on the tensorization technique, provided an appropriate choice of univariate bases is made \cite{bachmayr2023approximation}. However, using tensorization has the advantage of removing the need to tailor the approximation tool to the specific smoothness. In addition, tensorization allows us to exploit additional structures that can not be captured otherwise (see the example of fractal functions mentioned above). In this work, we have chosen to analyze a particular approximation tool which relies on  a resolution-wise ordering of digits (after tensorization) and the use of the TT format. With the same ordering of variables, near-optimal rates could be obtained with other tree tensor networks, such as the HT format (balanced tree). This can be deduced from results on the conversion from the TT format to the HT format \cite{Buczynska2015Oct,Buczynska2020Jan}. Optimal rates for classical smoothness classes could also be obtained with a coordinate-wise ordering of digits and the TT format, or more general tree tensor networks. However, the corresponding approximation spaces may not possess the nice linear-space structure of the approximation spaces of TTs with resolution-wise ordering.  
    One could also consider an approximation tool based on tree tensor networks where the dimension partition tree is a free parameter, as in \cite{Michel2022Bernoulli}. This leads to a much larger approximation space, but highly nonlinear. Even if it yields highly nonlinear approximation spaces, allowing  different orderings of variables or different TNs topologies may be crucial for high-dimensional problems. For instance, \cite{Grelier2022} presents examples of high-dimensional functions that can be approximated with a dimension-independent rate of convergence using a suitable ordering of variables, whereas a poor choice of  ordering leads to the curse of dimensionality.
        }
    \\\par
        \paragraph{\bfseries Functions over arbitrary domains.} 
      \revY{Finally, we believe that the above approximation results could} be extended to
            bounded domains $\Omega\subset\R^D$ with
            Lipschitz boundary or, more generally,
            $(\varepsilon,\delta)$-domains using
            bounded extension operators as in, e.g.,
            \cite{devore1993besov}, or \cite{ali2020approximation} for neural networks approximation. \revY{This is deferred to future investigation}.


%% file: review.tex

\section{Review of Smoothness Classes}\label{sec:review}
In this section, we review
Besov spaces, 
characterizations of Besov spaces by spline systems,
Besov embeddings and best rates of linear and nonlinear approximation
with splines. \revQ{These results are taken from \cite{devoresplines,devore1988interpolation,Leisner,Hansen2012,NonCompact}}.

\subsection{Besov Spaces}\label{sec:besov}
The main result of this work concerns spaces of isotropic,
anisotropic and mixed smoothness.
In this subsection,
we define classes of Besov spaces for each type
of smoothness. These Besov
spaces will serve as prototypes for our results
but, in principle, one could consider
other types of smoothness classes.


\subsubsection{Isotropic Besov Spaces}
Let $\Omega\subset\R^D$
be a bounded Lipschitz domain and $f\in L^p(\Omega)$ for
$0<p\leq\infty$. For $h\in\R^D$,
we denote by $\tau_h:L^p(\Omega)\rightarrow L^p(\Omega_h)$
the translation operator
$(\tau_h f)(x):=f(x+h)$,
$
    \Omega_h:=\{x\in\Omega:\;x+h\in\Omega\}.
$
Define the $r$-th difference as
\begin{align*}
    \Delta_h^r:=(\tau_h-\id)^r:=
    \underbrace{(\tau_h-\id)\circ\ldots\circ(\tau_h-\id)}_{r\text{ times}}
    :L^p(\Omega)\rightarrow L^p(\Omega_{rh}).
\end{align*}
Let $|h|_\alpha$ denote the standard
$\alpha$-(quasi-)norm on $\R^D$ for
$0<\alpha\leq\infty$.
The \emph{isotropic} modulus of smoothness is defined
for any $t>0$ as
\begin{align*}
    \omega_r(f,t)_p:=\sup_{0<|h|_2\leq t}
    \norm{\Delta_h^r(f)}[p].
\end{align*}
The \emph{isotropic} Besov (quasi-)semi-norm is defined for any
$0<p,q\leq\infty$ and any $s_\I>0$ and $r:=\lfloor s_\I\rfloor+1$
as
\begin{align}\label{eq:defisonorm}
    \snorm{f}[\B]:=
    \begin{cases}
        \left(\int_0^1[t^{-s_\I}\omega_r(f,t)_p]^q\frac{\d t}{t}\right)^{1/q},&\quad 0<q<\infty,\\
        \sup_{t>0}t^{-s_\I}\omega_r(f,t)_p,&\quad q=\infty,
    \end{cases}
\end{align}
and the (quasi-)norm as
\begin{align*}
    \norm{f}[\B]:=\norm{f}[p]+\snorm{f}[\B].
\end{align*}
The \emph{isotropic} Besov space is defined as
\begin{align*}
    \B:=\left\{f \revZ{\in} L^p(\Omega):\;
    \norm{f}[\B]<\infty\right\}.
\end{align*}


\subsubsection{Anisotropic Besov Spaces}
Let $e_\nu\in\R^D$ be the $\nu$-th canonical vector and define
the $\nu$-th coordinate difference for
$h\in\R_{>0}$ as
\begin{align*}
    \Delta_h^{r,\nu}(f):=\Delta^r_{he_\nu}\revZ{(f)}.
\end{align*}
The corresponding $\nu$-th modulus of smoothness
is defined for any $t>0$ as
\begin{align*}
    \omega_r^\nu(f,t)_p:=\sup_{0<h\leq t}
    \norm{\Delta_h^{r,\nu}(f)}[p].
\end{align*}
The \emph{anisotropic} Besov (quasi-)semi-norm for
$\bs\alpha:=(\alpha_1,\ldots,\alpha_D)$,
$\bar\alpha:=\max_\nu\alpha_\nu$,
$\underline\alpha:=\min_\nu\alpha_\nu$,
$r:=\lfloor\bar\alpha\rfloor+1$ and
$0<p,q\leq\infty$ is defined as
\begin{align*}
    \snorm{f}[\AB]:=\sum_{\nu=1}^D
    \left(\int_0^1[t^{-\alpha_\nu}\omega_r^{\nu}(f,t)_p]^q\frac{\d t}{t}\right)^{1/q},
\end{align*}
with the usual modification for $q=\infty$
as in \cref{eq:defisonorm},
and the corresponding (quasi-)norm
\begin{align*}
    \norm{f}[\AB]:=\norm{f}[p]+\snorm{f}[\AB].
\end{align*}
The \emph{anisotropic} Besov space is then defined accordingly
as the space of $L^p$-functions with finite norm.
We will also require the following aggregated
smoothness parameter
\begin{align*}
    s_\A:=s_\A(\alpha_1,\ldots,\alpha_D):=
    D\left(\alpha_1^{-1}+\ldots+\alpha_D^{-1}\right)^{-1}.
\end{align*}
\rev{In the special case $\boldsymbol{\alpha} = (s_I, \dots, s_I)$, it holds $s_A = \underline{\alpha} = \bar \alpha =  s_I$, and the space $\AB$ coincides with the isotropic Besov space $\B$. }

\subsubsection{Besov Spaces of Mixed Dominating Smoothness}
Let $\beta\subset\{1,\ldots,D\}$ and $h=(h_1,\ldots,h_D)\in\R^D_{>0}$.
Using the $\nu$-th coordinate difference $\Delta_{h_\nu}^{r,\nu}$ defined
above, we set
\begin{align*}
    \Delta_h^{r,\beta}(f):=\left(\bigotimes_{\nu\in\beta}\Delta_{h_\nu}^{r,\nu}\right)(f).
\end{align*}
Then, for $t\in\R^D_{>0}$, we define the
\emph{mixed} modulus of smoothness
\begin{align*}
    \omega_r^\beta(f,t)_p:=\sup_{0<h<t}\norm{\Delta_h^{r,\beta}(f)}[p],
\end{align*}
where $0<h<t$ is meant component-wise.
For $0<p,q\leq\infty$,
$s_\M>0$ and $r:=\lfloor s_\M\rfloor$+1, we define
the (quasi-)semi-norm
\begin{align*}
      \snorm{f}[MB^{s_\M,\beta}_q(L^p(\Omega))]:=
    \left(\int_{[0,1]^D}
    \left[\left\{\prod_{\nu\in\beta}t_\nu\right\}^{-s_\M}\omega_r^\beta(f,t)_p
    \right]^q\frac{\d t}{\prod_{\nu\in\beta}t_\nu}\right)^{1/q},
\end{align*}
with the standard modification for $q=\infty$.
The \emph{mixed} Besov (quasi-)norm is then defined as
\begin{align*}
    \norm{f}[\MB]:=\norm{f}[p]+
    \sum_{\beta\subset\{1,\ldots, D\}}\snorm{f}[MB^{s_\M,\beta}_q(L^p(\Omega))],
\end{align*}
and the corresponding \emph{mixed} Besov space
$\MB$ accordingly.

\subsection{Classical Results on Besov Spaces}
Using the decomposition defined in \cref{eq:decomp} we can characterize
the Besov norm and correspondingly Besov spaces.
\begin{theorem}[Characterization of Besov Spaces
\cite{devore1988interpolation,Leisner,Hansen2012,NonCompact}]\label{thm:chars}
    Let $\Omega:=[0,1)^D$.
    \begin{enumerate}[label=(\roman*)]
        
        \item Let $0<p,q\leq\infty$ and either $0<\bar\alpha<
        \min\{\bar m+1,\bar m+1/p\}$,
        or $\bar\alpha=\min\{\bar m+1,\bar m+1/p\}$
        and $q=\infty$. Then,
        $f\in\AB$ if and only if \revZ{
        $f = \sum_{l=0}^\infty \sum_{\vert \lambda_\A\vert = l} d_{\lambda_\A,p}(f) \varphi_{\lambda_\A} $
        with convergence} 
        in
        $L^p$ and
        \begin{align*}
            \norm{f}[\AB]\sim
            \left(\sum_{l=0}^\infty b^{\underline{\alpha} ql}\left[\sum_{|\lambda_\A|=l}
            |d_{\lambda_\A,p}(f)|^p\right]^{q/p}
            \right)^{1/q}<\infty
        \end{align*}
        and the usual modification for $q=\infty$. \rev{In the special case $\boldsymbol{\alpha} = (s_\I, \dots, s_\I)$, it holds $s_\A = \underline{\alpha} = \bar \alpha= s_\I$ and $AB^{\bs\alpha}_q(L^p(\Omega))$ coincides with the isotropic Besov space $B^{s_\I}_q(L^p(\Omega))$.}
        
        \item Let $0<p,q\leq\infty$ and either $0<s_\M<
        \min\{\bar m+1,\bar m+1/p\}$,
        or $s_\M=\min\{\bar m+1,\bar m+1/p\}$ and $q=\infty$. Then,
        $f\in\MB$ if and only if 
        \revZ{
        $f = \sum_{l=0}^\infty \sum_{\vert \lambda_\M \vert_1 = l} d_{\lambda_\M,p}(f) \varphi_{\lambda_\M} $
        with convergence} 
        in
        $L^p$ and
        \begin{align*}
            \norm{f}[\MB]\sim
            \left(\sum_{l\in(\N_{\geq 0})^D} b^{s_\M q|l|_1 }\left[\sum_{|\lambda_\M|=l}
            |d_{\lambda_\M,p}(f)|^p\right]^{q/p}
            \right)^{1/q}<\infty
        \end{align*}
        and the usual modification for $q=\infty$. 
       
    \end{enumerate}
\end{theorem}

The above characterizations can be used to
infer the following embeddings.
\begin{theorem}[Besov Embeddings \cite{devore1988interpolation,Leisner,Hansen2012,NonCompact}]\label{thm:embeds}
    Let $\Omega:=[0,1)^D$.
    \begin{enumerate}[label=(\roman*)]
        
        \item For the anisotropic Besov space we have the continuous
        embeddings
        \begin{align*}
            AB^{\bs\alpha}_q(L^\tau)\hookrightarrow L^p(\Omega),
        \end{align*}
        for $0<p<\infty$, $0<q\leq\tau$, $s_\A>0$ and
        $0<\tau<p$ such that
        \begin{align*}
            s_\A/D\geq \frac{1}{\tau}-\frac{1}{p}.
        \end{align*}
\rev{In the special case $\boldsymbol{\alpha} = (s_\I, \dots, s_\I)$, it holds $s_\A = \underline{\alpha} = \bar \alpha= s_\I$ and $AB^{\bs\alpha}_q(L^\tau)$ coincides with the isotropic Besov space $B^{s_\I}_q(L^\tau)$.}
        
        \item For the mixed Besov space we have the continuous
        embeddings
        \begin{align*}
            MB^{s_\M}_q(L^\tau)\hookrightarrow L^p(\Omega),
        \end{align*}
        for $1<p<\infty$, $0<q\leq\tau$, $s_\M>0$ and
        $0<\tau<p$ such that
        \begin{align*}
            s_\M\geq \frac{1}{\tau}-\frac{1}{p}.
        \end{align*}
    \end{enumerate}
\end{theorem}

As a linear method we consider approximating
a target function by a sum of all dilated splines on a given level,
where the sum is given by the
quasi-interpolator of the near-best polynomial
projection from \cref{eq:decomp}.

\begin{theorem}[Linear Approximation Rates
\cite{devoresplines,Leisner,Hansen2012,NonCompact}]\label{thm:classiclinear}
    Let $\Omega:=[0,1)^D$.
    \begin{enumerate}[label=(\roman*)]
        
        \item Let $0<p,q\leq\infty$ and $0<\bar\alpha\leq
        \min\{\bar m+1,\bar m+1/p\}$.
        Let $f\in AB^{\bs\alpha}_p(L^p(\Omega))$ and, for $l\in\N$,
        set
        $\varphi_n:=
        \sum_{|\lambda_\A|\leq l}d_{\lambda_\A,p}(f)\varphi_{\lambda_\A}$
        with the number of terms in the sum bounded at most by a constant
        multiple of
        $n:=b^{lD\underline{\alpha}/s_\A}$. Then,
        \begin{align*}
            \norm{f-\varphi_n}[L^p]\lesssim
            b^{-\underline{\alpha} l} \snorm{f}[AB^{\bs\alpha}_p(L^p(\Omega))]
            =
            n^{-s_\A/D} \snorm{f}[AB^{\bs\alpha}_p(L^p(\Omega))].
        \end{align*}
        \rev{In the special case $\boldsymbol{\alpha} = (s_\I, \dots, s_\I)$, it holds $s_\A = \underline{\alpha} = \bar \alpha= s_\I$ and $\AB$ coincides with the isotropic Besov space $\B$.}
        \item Let $0<p,q\leq\infty$ and $0<s_\M\leq
        \min\{\bar m+1,\bar m+1/p\}$.
        Let $f\in MB^{s_\M}_p(L^p(\Omega))$ and, for $l\in \N^D$,
        set
        $\varphi_n:=
        \sum_{|\lambda_\M|_1\leq l}d_{\lambda_\M,p}(f)\varphi_{\lambda_\M}$
        with the number of terms in the sum bounded at most by a constant
        multiple of
        $n:=L^{D-1}b^{L}$. Then,
        \begin{align*}
            \norm{f-\varphi_n}[L^p]\lesssim
            b^{-s_\M L} \snorm{f}[MB^{s_\M}_p(L^p(\Omega))]
            \leq
            n^{-s_\M}(\log_b(n))^{s_\M(D-1)}\snorm{f}[MB^{s_\M}_p(L^p(\Omega))].
        \end{align*}
    \end{enumerate}
\end{theorem}

Finally, the characterization together with the Besov embeddings
imply the following rates for the best $n$-term approximation.
\begin{theorem}[Best $n$-Term Approximation
\cite{devoresplines,Leisner,Hansen2012,NonCompact}]\label{thm:classicbestnterm}
    Let $\Omega:=[0,1)^D$.
    \begin{enumerate}[label=(\roman*)]
        
        \item Let $0<p<\infty$, $0<\bar\alpha<\min\{\bar m+1,\bar m+1/p\}$
        and $0<\tau<p$
        such that
        \begin{align*}
            s_\A/D\geq 1/\tau-1/p.
        \end{align*}
        Then, for any $f\in AB^{\bs\alpha}_q(L^\tau(\Omega))$
        with $q\leq\tau$, it holds
        \begin{align*}
            E_n(f, \Phi_\A)_{L^p}\lesssim n^{-s_\A/D}\snorm{f}[AB^{\bs\alpha}_{q}(L^\tau(\Omega))].
        \end{align*}
\rev{In the special case $\boldsymbol{\alpha} = (s_\I, \dots, s_\I)$, it holds $s_\A = \underline{\alpha} = \bar \alpha= s_\I$ and $AB^{\bs\alpha}_{q}(L^\tau(\Omega))$ coincides with the isotropic Besov space $B^{s_\I}_{q}(L^\tau(\Omega))$.}

        \item Let $1<p<\infty$, $0<s_\M<\min\{\bar m+1,\bar m+1/p\}$
        and $0<\tau<p$
        such that
        \begin{align*}
            s_\M\geq 1/\tau-1/p.
        \end{align*}
        Then, for any $f\in MB^{s_\M}_q(L^\tau(\Omega))$
        with $q\leq\tau$, it holds
        \begin{align*}
            E_n(f, \Phi_\M)_{L^p}\lesssim n^{-s_\M}
            (\log_b(n))^{s_\M(D-1)}\snorm{f}[B^{s_\M}_{q}(L^\tau(\Omega))].
        \end{align*}
    \end{enumerate}
\end{theorem}
